\documentclass[11pt,letterpaper]{article}
\usepackage{geometry}
\def\PrintMode{0}%normally PrintMode = 0. If PrintMode = 1, extreme methods will be used to reduce the number of pages. Can also set fontsize = 10pt to further shrink pages...

\usepackage{tikz}
\usetikzlibrary{positioning}
\usepackage{pgffor}
\usetikzlibrary{decorations.pathreplacing}
\usepackage{expl3}

\usepackage[T1]{fontenc}
\usepackage{graphicx}
\usepackage{float}
\usepackage[lined,ruled,linesnumbered]{algorithm2e}
\usepackage{xspace}
\usepackage[utf8]{inputenc}
\usepackage[noadjust]{cite}
\usepackage{amsmath, amssymb, amsfonts, amsthm}
\usepackage{enumerate}
\usepackage{xcolor}
\usepackage[linktocpage=true,pagebackref=true,colorlinks,linkcolor=magenta,citecolor=blue,bookmarks,bookmarksopen,bookmarksnumbered]{hyperref}
\usepackage{bm}
\usepackage{bbm}
\usepackage[normalem]{ulem}
\usepackage{paralist}
\usepackage{gensymb}
\usepackage{thmtools}
\usepackage{rotating}
\usepackage{mdframed}
\usepackage{multirow}
\usepackage{tabularx}
\usepackage{verbatim}
\usepackage{mathrsfs}
\usepackage{indentfirst}
\usepackage{mleftright}
\usepackage[nottoc]{tocbibind}%add bib in toc
\usepackage{complexity}
\usepackage{tablefootnote}
\usepackage{epigraph}
\usepackage{autobreak}

\ifnum\PrintMode=1

\usepackage{savetrees}
\else
\fi

\graphicspath{{./figs/}}

\usepackage[capitalize]{cleveref}%cleveref must be loaded *last*

\newtheorem{theorem}{Theorem}[section]
\newtheorem{lemma}[theorem]{Lemma}
\newtheorem{proposition}[theorem]{Proposition}
\newtheorem{corollary}[theorem]{Corollary}

\newtheorem{observation}[theorem]{Observation}
\newtheorem{fact}[theorem]{Fact}

\theoremstyle{definition}
\newtheorem{definition}[theorem]{Definition}

\newtheorem{remark}[theorem]{Remark}

\renewcommand*\backref[1]{\ifx#1\relax \else (cit.~on p.~#1) \fi} %http://latex.org/forum/viewtopic.php?t=3670

\makeatletter
\def\moverlay{\mathpalette\mov@rlay}
\def\mov@rlay#1#2{\leavevmode\vtop{%
		\baselineskip\z@skip \lineskiplimit-\maxdimen
		\ialign{\hfil$\m@th#1##$\hfil\cr#2\crcr}}}
\newcommand{\charfusion}[3][\mathord]{
	#1{\ifx#1\mathop\vphantom{#2}\fi
		\mathpalette\mov@rlay{#2\cr#3}
	}
	\ifx#1\mathop\expandafter\displaylimits\fi}
\makeatother

\renewcommand{\poly}{\mathrm{poly}}

\renewcommand{\polylog}{\mathrm{polylog}}

\newlang{\MCSP}{MCSP}
\newlang{\MFSP}{MFSP}
\newlang{\MKtP}{MKtP}
\newlang{\MKTP}{MKTP}
\newlang{\itrMCSP}{itrMCSP}
\newlang{\itrMKTP}{itrMKTP}
\newlang{\itrMINKT}{itrMINKT}
\newlang{\MINKT}{MINKT}
\newlang{\MINK}{MINK}
\newlang{\MINcKT}{MINcKT}
\newlang{\CMD}{CMD}
\newlang{\DCMD}{DCMD}
\newlang{\CGL}{CGL}
\newlang{\PARITY}{PARITY}
\newlang{\Empty}{\textsc{Empty}}
\newlang{\Avoid}{\textsc{Avoid}}
\newlang{\Sparsification}{\textsc{Sparsification}}
\newlang{\HamEst}{\mathsf{HammingEst}}
\newlang{\HamHit}{\mathsf{HammingHit}}
\newlang{\CktEval}{\textsc{Circuit-Eval}}
\newlang{\Hard}{\textsc{Hard}}
\newlang{\cHard}{\textsc{cHard}}
\newlang{\CAPP}{CAPP}
\newlang{\GapUNSAT}{GapUNSAT}
\newlang{\OV}{OV}
\renewlang{\PCP}{PCP}
\newlang{\PCPP}{PCPP}
\newclass{\Avg}{Avg}
\newclass{\ZPEXP}{ZPEXP}
\newclass{\DLOGTIME}{DLOGTIME}
\newclass{\ALOGTIME}{ALOGTIME}
\newclass{\ATIME}{ATIME}%alternating time
\newclass{\SZKA}{SZKA}
\newclass{\Laconic}{Laconic\text{-}}
\newclass{\APEPP}{APEPP}
\newclass{\SAPEPP}{SAPEPP}
\newclass{\TFSigma}{TF\Sigma}
\newclass{\NTIMEGUESS}{NTIMEGUESS}

\newlang{\Formula}{Formula}
\newlang{\THR}{THR}
\newlang{\EMAJ}{EMAJ}
\newlang{\MAJ}{MAJ}
\newlang{\SYM}{SYM}
\newlang{\DOR}{DOR}
\newlang{\ETHR}{ETHR}
\newlang{\Midbit}{Midbit}
\newlang{\LCS}{LCS}
\newlang{\TAUT}{TAUT}
\newlang{\Poly}{\text{-}Poly}

\newcommand{\dgr}{\mathsf{deg}}

\renewcommand{\epsilon}{\varepsilon}

\usepackage[draft, inline,marginclue,index]{fixme}  %Remove draft to erase all notes, remove in line to erase inline notes.
\fxsetup{theme=color,mode=multiuser}

\definecolor{color1}{RGB}{46,134,193}
\FXRegisterAuthor{hanlin}{ahanlin}{\colorbox{color1}{\color{white}Hanlin}}

\definecolor{color7}{RGB}{128,0,128}
\FXRegisterAuthor{yeyuan}{ayeyuan}{\colorbox{color7}{\color{white}Yeyuan}}

\definecolor{color3}{RGB}{255,128,0}
\FXRegisterAuthor{jiatu}{ajiatu}{\colorbox{color3}{\color{white}Jiatu}}

\definecolor{color4}{RGB}{150,150,150}
\FXRegisterAuthor{tianqi}{atianqi}{\colorbox{color4}{\color{white}Tianqi}}

\definecolor{color2}{RGB}{20,60,100}
\FXRegisterAuthor{zhikun}{azhikun}{\colorbox{color2}{\color{white}Zhikun}}

\definecolor{color6}{RGB}{250,0,250}
\FXRegisterAuthor{yizhi}{ayizhi}{\colorbox{color6}{\color{white}Yizhi}}

\definecolor{color5}{RGB}{128,128,128}
\FXRegisterAuthor{task}{atask}{\colorbox{color5}{\color{white}Task}}

\begin{document}

\title{Unique-neighbor Expanders with Better Expansion for Polynomial-sized Sets}
\author{
Yeyuan Chen\thanks{Department of EECS, University of Michigan, Ann Arbor. \href{mailto:yeyuanch@umich.edu}{\texttt{yeyuanch@umich.edu}}}} 
\date{}
\maketitle
\begin{abstract}
A $(d_1,d_2)$-biregular bipartite graph $G=(L\cup R,E)$ is called left-$(m,\delta)$ unique-neighbor expander iff each subset $S$ of the left vertices with $|S|\leq m$ has at least $\delta d_1|S|$ unique-neighbors, where unique-neighbors mean vertices with exactly one neighbor in $S$. We can also define right/two-sided expanders similarly. In this paper, we give the following three strongly explicit constructions of unique-neighbor expanders with better unique-neighbor expansion for polynomial-sized sets, while sufficient expansion for linear-sized sets is also preserved:
\begin{itemize}
\item Two-sided $(n^{1/3-\epsilon},1-\epsilon)$ lossless expanders for arbitrary $\epsilon>0$ and aspect ratio.
\item Left-$(\Omega(n),1-\epsilon)$ lossless expanders with right-$(n^{1/3-\epsilon},\delta)$ expansion for some $\delta>0$.
\item Two-sided-$(\Omega(n),\delta)$ unique-neighbor expanders with two-sided-$(n^{\Omega(1)},1/2-\epsilon)$ expansion.
\end{itemize}
The second construction exhibits the first explicit family of one-sided lossless expanders with unique-neighbor expansion for polynomial-sized sets from the other side and constant aspect ratio. The third construction gives two-sided unique-neighbor expanders with additional $(1/2-\epsilon)$ unique-neighbor expansion for two-sided polynomial-sized sets, which approaches the 
$1/2$ requirement in Lin and Hsieh (arXiv:2203.03581).

Our techniques involve tripartite product recently introduced by Hsieh et al (STOC 2024), combined with a generalized existence argument of biregular graph with optimal two-sided unique-neighbor expansion for almost all degrees. We also use a new reduction from large girth/bicycle-freeness to vertex expansion, which might be of independent interest.
\end{abstract}
\setcounter{tocdepth}{2}
%how to reduce the inverse pair and make it work for all degrees
%use of edge-vertex incidence graph in noga's paper
%fix a bug in thm5.1 HMMP24
%explain left/right/biregular
%need an image in def4.5
%add proof overview for long proofs
%generalize random biregular graph profile 
%explain parameter dependency in proof
%n^1/3 is tight in some sense from MM21
%mu independent with epsilon in thm3
%generalize biregular graph profile

%why strong explicit construction is useful.
%what is ramanujan graph
%explain balanced
%emphasize advantage of new girth reduction
%emphasize difficulty of generalizing biregular graph profile.
\pagenumbering{roman}

\pagenumbering{arabic}
\section{Introduction}
A one-sided unique-neighbor expander is a bipartite graph such that any small subset $S$ of its left vertices have many unique-neighbors, where unique-neighbors mean right vertices that are incident to exactly one vertex in $S$. Furthermore, if we can guarantee $(1-\epsilon)$ fraction of edges from $S$ entering unique-neighbors for arbitrarily small constant $\epsilon>0$, we call it a lossless expander. Explicit constructions of one-sided unique-neighbor expanders have many applications in some areas of theoretical computer science, such as error-correcting codes \cite{expandercode,ldpc}, compressed sensing \cite{compressed,compressed2}, proof complexity \cite{widgerson,widgerson2} and pseudorandomness \cite{guv,extcode}. Motivated by the above applications, people did a long line of work to explicitly construct one-sided unique-neighbor (lossless) expanders \cite{uniqueneighbor,zigzag,louis24,dinur23,noga23,hdx}. These constructions can guarantee one-sided lossless expansion for any $O(n)$-size vertex sets.

A recent work \cite{lin2022good} found a reduction that explicit constructions of two-sided lossless expanders with specific algebraic property yield quantum LDPC codes with linear-time decoder. This motivates people to focus more on constructions of two-sided unique-neighbor expanders, which means the graph should have good expansion for all small sets from \emph{both} sides. Unfortunately, all classical constructions of one-sided unique-neighbor expanders don't exhibit expansion from the other side. Some constructions using the so-called `Routed Product' technique will even destroy the expansion of even constant-sized sets from the other side.

Unlike unique-neighbor expanders, we know a lot of explicit constructions of spectral expanders \cite{lps88,lpsimprove,everydegree,ow20} whose adjacency matrices have bounded second-largest eigenvalues. Although it was pointed out that mere spectral expanders don't exhibit unique-neighbor expansion, many previous constructions \cite{louis24,dinur23,noga23} use them as a `base graph'. In very recent work, \cite{hmmp24} generalizes this idea and gets explicit construction of two-sided unique-neighbor expanders. However, their construction only guarantees a small constant fraction of edges entering unique-neighbors, while \cite{lin2022good} requires this fraction to be larger than $1/2$. This motivates us to construct two-sided expanders with `better than just unique-neighbor' or even lossless expansion.

In this paper, we make progress on this topic. We give strongly explicit constructions of expanders with better two-sided expansion for polynomial-sized sets. To the best of our knowledge, these are the first set of constructions that give better two-sided expansion for polynomial-sized sets, while simultaneously preserving unique-neighbor/lossless expansion for linear-sized sets. Our results are precisely described below.
\subsection{Our Results}
For any graph $G$ and its vertex set $S\subseteq V(G)$, let $\mathsf{UN}_G(S)$ denote the set of unique-neighbors of $S$ in $G$ defined as $\mathsf{UN}_G(S):=\{v\colon e(v,S)=1,v\in V(G)\backslash S\}$. Here, $e(v,S)$ denotes the number of edges between $S$ and vertex $v$.

For an infinite family of graphs $(G_n)_n$, if there is an algorithm $A(1^n)$ that outputs $G_n$ in $\poly(n)$ time, we call it an explicit construction of $G_n$. Moreover, if there is an algorithm $B(1^n,v,i)$ that outputs the $i$-th edge incident to $v\in V(G_n)$ in $\polylog(n)$ time, we call it a strongly explicit construction of $G_n$.
\begin{definition}
For any $(d_1,d_2)$-biregular bipartite graph $G=(L\cup R,E)$, $\beta:=\frac{d_1}{d_2}=\frac{|R|}{|L|}$ is called its aspect ratio. For any $0<\delta<1$ and $m$, we define:

1. If for any $S\subseteq L$ with $|S|\leq m$, we have $|\mathsf{UN}_G(S)|\ge \delta d_1|S|$, we call $G$ a left-$(m,\delta)$ unique-neighbor expander.

2. If for any $S\subseteq R$ with $|S|\leq m$, we have $|\mathsf{UN}_G(S)|\ge \delta d_2|S|$, we call $G$ a right-$(m,\delta)$ unique-neighbor expander.
\end{definition}
When $\delta=1-\epsilon$ for arbitrary small constant $\epsilon$, we call it a lossless expander.
\paragraph{Unbalanced Two-sided Lossless Expander for Polynomial-sized Sets}
We observe that large girth implies lossless expansion (See also \cite{hmmp24,kahale95}). For the best known explicit construction of biregular large girth graph with girth about $\frac{4}{3}\log{n}$ (\cite{lps88}), it implies $(n^{1/3-\epsilon},1-\epsilon)$ lossless expansion. However, many of similar constructions are not known to be biregular and  
 unbalanced (i.e. with aspect ratio $\beta\neq 1$) at the same time. In our first result, we get a strongly explicit construction for biregular bipartite two-sided lossless expanders matching the above bound with arbitrary bidegrees and aspect ratios by considering an adaption of \cite{lazebnik1}:
\begin{theorem}[\Cref{thmlaze}, Informal]
For all $\epsilon,\beta>0$, there are infinitely many bidegrees $d_1,d_2$ where $\frac{d_1}{d_2}=\beta$, such that we have strongly explicit construction of an infinite family of two-sided $(d_1,d_2)$-biregular $(\Omega(n^{1/3-\epsilon}),1-\epsilon)$ unique-neighbor (lossless) expanders.
\end{theorem}
\paragraph{Lossless Expanders with Polynomial-sized Expansion on Another Side}
There is also a long line of research working on constructing one-sided lossless expanders for linear-sized sets (e.g. \cite{zigzag,louis24}). However, to the best of our knowledge, if we require lossless expansion for $\Omega(n)$-size sets from one side, all known constructions give little guarantee of expansion from the other side. Some constructions using routed product (\cite{louis24,dinur23,noga23}) will inevitably destroy the expansion of even constant-sized sets from the other side. Even if we allow the aspect ratio $\beta>1$, the best-known expansion from the other side (\cite{ow20} or generalized \cite{hmmp24}) only guarantees expansion for $\exp(O(\sqrt{\log{n}}))\leq n^{o(1)}$-sized sets, which is far from satisfying. In this work, we give the first strongly explicit construction of left-lossless expanders for linear-sized sets that also guarantees right-unique-neighbor expansion for polynomial-sized sets. 
\begin{theorem}[\Cref{thmlosslessunique}, Informal]
For any $\epsilon>0$, there exists $\delta(\epsilon)>0,\beta(\epsilon)>1$ such that for infinitely many $d_1,d_2$ where $\frac{d_1}{d_2}=\beta$, there is strongly explicit construction of an infinite family of $(d_1,d_2)$-biregular graphs $G$ such that:
\begin{compactitem}
\item[(1)] $G$ is a left-$(\Omega(n),1-\epsilon)$ lossless expander.

\item[(2)] $G$ is a right-$(n^{1/3-\epsilon},\delta)$ unique-neighbor expander.

\item[(3)] $G$ is a right-$(n^{o(1)},1-\epsilon)$ lossless expander.
\end{compactitem}
\end{theorem}
\begin{remark}
Fix $\epsilon$, the aspect ratio $\beta$ can be seen as a constant for arbitrarily large bidegree $d_1,d_2$. However, a weakness of this result is that we require $\beta>1$, which is not satisfying in some applications.
\end{remark}
\paragraph{Two-sided Unique-neighbor Expanders with Better Expansion for Polynomial-sized Sets} Motivated by potential constructions of good quantum LDPC codes with linear time decoder via two-sided expanders (\cite{lin2022good}), 
\cite{hmmp24} gives a construction of two-sided-$(\Omega(n),\delta)$ unique-neighbor expanders. However, the parameter $\delta$ in their construction must be some small enough constant $\delta<0.005$, which doesn't yield satisfying quantum LDPC codes via \cite{lin2022good}'s reduction. In fact, the interested case in \cite{lin2022good} requires $\delta>1/2$. This motivates us to care about unique-neighbor expansion with parameter $\delta=1/2$. In our third result, we give a strongly explicit construction of two-sided $(\Omega(n),\delta)$-unique-neighbor expanders that also guarantees $(n^{\Omega(1)},  1/2-\epsilon)$ two-sided unique-neighbor expansion for arbitrarily small $\epsilon>0$. 
\begin{theorem}[\Cref{thmtwoside} Informal] 
For any $\epsilon,\beta>0$, there exists $\delta(\beta)>0$ such that for infinitely many $d_1,d_2$ where $\frac{d_1}{d_2}=\beta$, there is strongly explicit construction of an infinite family of $(d_1,d_2)$-biregular graphs $G$ such that:
\begin{compactitem}
\item[(1)] $G$ is a two-sided-$(\Omega(n),\delta)$ unique-neighbor expander.

\item[(2)] $G$ is a two-sided-$(n^{\Omega(1)},1/2-\epsilon)$ unique-neighbor expander.
\end{compactitem}
\end{theorem}
\paragraph{New Reduction from Large Girth to Expansion}
In the above construction, one key lemma we use is a reduction from large girth or bicycle-freeness to vertex expansion. This is also observed by some previous work \cite{hmmp24,reductionmohanty}. While previous proof uses spectral techniques, our proof builds on combinatorial arguments via subsampling. Furthermore,  Our reduction gives better bounds than previous work in two facets: (1). The size bound has an additional $\Omega(g)$ factor, where $g$ denotes the girth of the graph. (2). Our reduction applies to a wider range of parameters, which is crucial in the proof of \Cref{thmtwoside}. See \Cref{girthcompare} for more details. Moreover, our reduction is strictly stronger than the similar reduction appeared in \cite{reductionmohanty}. That's because \cite{reductionmohanty} requires the large-girth graph to be a (near) Ramanujan graph, while our new reduction only cares about large-girth itself modulo any spectral requirement.
\begin{theorem}[\Cref{girthtoexpansion} Informal]
For any $\epsilon,g,d>0$, it holds that: For any bipartite undirected graph $G=(L\cup R,E)$ that has girth $g$ and average left-degree $d$, if $|L|\leq O(g(\epsilon d)^{g/4})$, then there is $|R|\ge (1-\epsilon)d|L|$.
\end{theorem}
\begin{remark}
We can not simply use classical irregular Moore bound \cite{irregularmoore} to get \Cref{girthtoexpansion} since it only yields about $O((4+4\epsilon)^{g/4})$ size bound, far smaller than ours in most of interesting cases. Even the more refined bipartite irregular Moore bound \cite{bipartitemoore} is not applicable since it requires a minimum degree of at least $2$, which isn't true in most vertex expansion cases. If one resorts to the `pruning' trick to remove vertices with degree $1$, \cite{bipartitemoore} will lose advantage over \cite{irregularmoore} since we can't guarantee left/right average degree to be non-decreasing separately after pruning. Therefore, it's necessary to build this reduction from scratch. 
\end{remark}

We list some related previous results in \cref{table:comparison} for comparison.

%citation, linear, polynomial, aspect ratio
\begin{table}[H]

\caption{Comparison with previous constructions. This table only includes constructions that guarantee at least $\Omega(|S|)$ unique-neighbors. There are some other constructions that only guarantee the existence of unique-neighbors. All these explicit constructions work for infinitely many constant bidegrees $(d_1,d_2)$.}\label{table:comparison}
\resizebox{\textwidth}{!}
{
\begin{tabular}{|c|c|c|c|}
\hline
citation & linear-sized expansion & sublinear-sized expansion & aspect ratio\\
\hline
\cite{lps88,lpsimprove} & -- & two-sided $(n^{1/3-\epsilon},1-\epsilon)$ & $1$ \\
\hline
\Cref{thmlaze} & -- & two-sided $(n^{1/3-\epsilon},1-\epsilon)$ & any $\beta>0$ \\
\hline
\cite{louis24,zigzag,hdx} & left-$(\Omega(n),1-\epsilon)$ & -- & any $\beta>0$ \\
\hline
\cite{ow20} & left-$(\Omega(n),1-\epsilon)$ & two-sided $(\exp(\Omega(\sqrt{\log{n}})),1-\epsilon)$ & $\beta(\epsilon)>1$ \\
\hline
\multirow{2}*{\cref{thmlosslessunique}}&\multirow{2}*{left-$(\Omega(n),1-\epsilon)$} & two-sided $(n^{1/3-\epsilon},\delta(\epsilon))$ & \multirow{2}*{$\beta(\epsilon)>1$}\\
~ & ~ & two-sided $(\exp(\Omega(\sqrt{\log{n}})),1-\epsilon)$ & ~ \\
\hline
\cite{hmmp24} & two-sided $(\Omega(n),\delta(\beta,\epsilon))$ & two-sided $(\exp(\Omega(\sqrt{\log{n}})),1-\epsilon)$ & any $\beta>0$ \\
\hline
\Cref{thmtwoside} & two-sided $(\Omega(n),\delta(\beta))$ & two-sided $(n^{\Omega(1)},1/2-\epsilon)$ & any $\beta>0$ \\ \hline
\end{tabular}
}
\end{table}
\subsection{A Related Work}
After this paper has been submitted, an independent work \cite{eshan24} was published online working on the same topic. \cite{eshan24} is based on multiplicity codes and generalized Hermite interpolation, which is different from our techniques. Their main results are also incomparable with ours. Concretely, their graph degrees are $\Theta(\polylog{n})$ rather than constant in our paper. Their aspect ratios are also non-constant. The merit of their work is that they get two-sided lossless expansion, but other parameter settings are incomparable with \cref{thmlaze,thmlosslessunique}.
\subsection{Technical Overview}\label{technicaloverview}
In \cref{sec:reduction}, we sketch proof of our new reduction from large girth or bicycle-freeness to vertex expansion using subsampling technique. With this reduction, we describe our constructions \cref{thmlaze,thmlosslessunique,thmtwoside} in \cref{sec:construction}. Finally, in \cref{sec:tool} we briefly introduce some technical tools used in these constructions and how we generalize them from previous work.
\subsubsection{Vertex Expansion from Large Girth or Bicycle-freeness}\label{sec:reduction}
Our first technical tool is to give a new vertex expansion theorem from large girth or bicycle-freeness. Here we give a proof overview of expansion from large girth, and expansion from bicycle-freeness follows from similar arguments. Now, suppose a bipartite graph has girth $g$. First, similar to \cite{bipartitemoore}, we observe that we can't have two different $g/2$-simple paths sharing the same starting and ending points, since otherwise, they will form a cycle with length at most $g$, which contradicts the girth requirement. From a simple pigeonhole principle, we can bound the number of such paths by $O(n^2)$. On the other hand, consider a graph $G$ with an arbitrary spanning tree $T$, every non-tree edge $e$ in $G$ yields a cycle larger than $g$, and deleting $e$ will delete $\Omega(g)$ simple $g/2$-paths at the same time. Therefore, by iteratively deleting non-tree edges, we can lower bound the number of simple $g/2$-paths by $\Omega(g(m-n))$, where $m$ is the number of edges. Now, consider any left vertex set $S$ with size $n$ and its neighbor set $T=N_G(S)$. Suppose $T$ doesn't expand well and only has size $|T|\leq (1-\epsilon)dn$, where $d$ is the average degree of $S$. We can sample a random subset $T'\subseteq T$ where each $t\in T$ is selected with independent $p$ probability. Consider the lower and upper bound of the expected number of simple $g/2$-paths in random induced subgraph $H=G[S\cup T']$, we can get an inequality $\Omega(g)\mathbb{E}[|E(H)|-|V(H)|]\leq n^2p^{g/2}$, where $p^{g/4}$ is the probability for each simple $g/2$-path in $G$ to be contained in $H$. By choosing $p$ wisely, we can get a lower bound on $n$. This implies if $n$ is not large enough, we must have at least $(1-\epsilon)dn$ vertex expansion.
\begin{remark}\label{girthcompare}
Our theorem has an additional $\Omega(g)$ factor in the bound, which is asymptotically better than the bound in \cite{hmmp24} in some scenarios. Please also note that here we cannot directly use the classical bipartite irregular Moore bound (\cite{bipartitemoore}) since its proof implicitly requires the minimum degree in $G[S\cup T]$ is at least $2$, which doesn't hold in our case. Our theorems also extend the parameter range to all $d\ge 2$, which is useful in this work since \Cref{thmtwoside} uses $(2,d)$-biregular graphs as base graphs. A similar argument in \cite{hmmp24} only holds for $d\ge 3$, and its proof technique uses some involved spectral method, which is different from our subsampling method here.
\end{remark}
\subsubsection{Constructions}\label{sec:construction}
\paragraph{Large Girth Biregular Graph with Arbitrary Bidegrees and Aspect Ratios}
In this paragraph we sketch the construction of \cref{thmlaze}. From above, we know every graph with a large girth yields good two-sided lossless expansion. The classical construction \cite{lps88} gives construction with girth $\frac{4}{3}\log{n}$, but it is balanced while the ideal construction should fit arbitrary large bidegrees and arbitrary aspect ratio $\beta>0$. To amend this, we introduce the graph family $D(k,q)$ defined from \cite{lazebnik1}. By observing its construction we can easily adapt it to almost arbitrary bidegrees by restricting it to only a subset of its edge index. However, it only yields $\log{n}$ girth. To make it match the large girth of \cite{lps88}, we use the analysis in \cite{lazebnik2} to extract one of its connected components. Focusing only on this single component, we reduce the number of vertices while preserving the girth, and therefore boost the girth to $\frac{4}{3}\log{n}$. However, restricted within a single component, it becomes less clear whether the construction is still strongly explicit. We use the polynomial inverse technique to show that even restricted within a single connected component, this construction is still strongly explicit. 
\paragraph{Tripartite Product}
The explicit construction in \Cref{thmlosslessunique,thmtwoside} both use the \emph{Tripartite Product} framework introduced in \cite{hmmp24}. It is formally defined in \Cref{tripartite}. We first set up three sets of vertices $L\cup M\cup R$ as left/middle/right sets. Then, we prepare two base bipartite graphs $G_1=(L\cup M,E_1)$ and $G_2=(M\cup R,E_2)$ and embed them onto the vertices. Suppose $G_1$ is $D_1$-right-regular and $G_2$ is $D_2$-left-regular, we then prepare a small `gadget graph' $G_0=(L_0\cup R_0,E_0)$ where $|L_0|=D_1,|R_0|=D_2$. The tripartite product is defined as graph $G=(L\cup R,E)$ whose edges are obtained by placing a copy of $G_0$ on each middle vertex and `merge' edges from $G_0,G_1$ and $G_2$.
\begin{remark}
The advantage of the Tripartite Product is that $G_0$ is usually a constant-size graph, so we can use exhaustive search to make $G_0$ as good as random graphs. In our applications here, base graphs $G_1,G_2$ will be large graphs with good spectral expansion, and $G_0$ will be the graph with expansion `as good as random biregular graphs'. This is the crux to transform purely weak spectral expansion of base graphs to strong unique-neighbor expansion or even lossless expansion. We will give a more concrete description below.  
\end{remark}
\paragraph{Lossless Expanders with Polynomial-sized Unique-neighbor Right Expansion} In this paragraph we give a proof overview of \Cref{thmlosslessunique}. We use \cite{lps88} to construct $(D,D)$-biregular Ramanujan graph $G_1=(L\cup M,E_1)$ with $\frac{4}{3}\log{n}$ girth, and use \cite{ow20} to construct $(D,\epsilon D)$-biregular near-Ramanujan graph $G_2=(M\cup R,E_2)$ with $\Omega(\sqrt{\log{n}})$ bicycle-freeness. In terms of gadget graph $G_0$, we need a $(d,d)$-biregular constant-sized graph with optimal two-sided expansion. However, to make use of spectral expansion of $G_2$, we need $d=\Omega(D)$ such that $d>\Omega(\sqrt{\epsilon D^2})$. This motivates us to generalize the existence argument of optimal biregular expander to wider parameter settings (See \Cref{biregularsec}). Our construction will be the tripartite product of $G_1,G_2,G_0$ described above. To show lossless expansion for any fixed $O(n)$-sized left vertex set $S$. We first consider the set $\mathcal{U}$ of all neighbors of $S$ in $G_1$. Let $h=C$ be a large enough constant, we split $\mathcal{U}$ into heavy neighbors $\mathcal{U}_h$ and light neighbors $\mathcal{U}_l$. Heavy neighbors are those incident to at least $h$ vertices in $S$ and light neighbors are all other neighbors. The first filter is to control the number of edges between $S$ and $\mathcal{U}_h$ in $G_1$. It's not hard to use the subgraph density lemma and derive that at most about $O(1/C)$ percentage of edges enter $\mathcal{U}_h$. For most of the edges entering $\mathcal{U}_l$, each of them expands about $d$ unique neighbors in the corresponding $G_0$ copy. On the other hand, each vertex in $\mathcal{U}_l$ has at least about $d>\Omega(\sqrt{\epsilon D^2})$ unique neighbors expanding to the right side. By using the subgraph density argument for spectral expander $G_2$, we derive that most of unique-neighbors in $\mathcal{U}_l$-copies of $G_0$ are also unique neighbors in the whole graph $G$. Combined with the above, we know \emph{most of} edges from $S$ in $G_1$ expand about $d$ unique-neighbors in $G_0$ copies. Furthermore, \emph{most of} these unique-neighbors continue to be unique-neighbors on the right side. These two claims guarantee lossless expansion. The unique-neighbor expansion from right to left follows from a similar argument. However, this time we don't have $d>\Omega(D)$ and cannot make use of the subgraph density argument of spectral expander $G_1$. Alternatively, we use the large girth of $G_1$ to guarantee that most of unique-neighbors of $G_0$ copies continue to be unique-neighbors in the left part. 
\begin{remark}
In the current framework, if one chooses the parameters that yield linear-sized lossless expansion from one side, then it is impossible to get even linear-sized unique-neighbor expansion from the other side if sticking to the same proof strategy. This relates to parameter trade-offs when using subgraph density arguments and the existence of optimal gadget graphs simultaneously. See more details in \Cref{losslesssec}. 
\end{remark}
\paragraph{Two-sided Unique-neighbor Expander with Polynomial-sized $(1/2-\epsilon)$ Expansion} We use the same tripartite product framework as in \Cref{thmlosslessunique}. However, this time we use edge-vertex incidence graphs of \cite{lps88} as the two base graphs. edge-vertex incidence graph of $G$ is formally defined in \Cref{edgevertexdef}. We show that edge-vertex incidence graph of a (near)-Ramanujan graph is still a (near)-Ramanujan graph. Another advantage is that it preserves the property of large girth or bicycle-freeness, which is used to derive expansion. We use a similar proof as above to show that from each side, 
$O(n)$-sized vertex set has a constant fraction of edges entering unique-neighbors. To show two-sided $(1/2-\epsilon)$ unique-neighbor expansion, we use the large girth condition and apply the reduction from girth to expansion for polynomial-sized sets. Note that we only get $(1/2-\epsilon)$ unique-neighbor expansion, not lossless expansion. That's because the edge-vertex incidence graph of a $d$-regular graph is a $(2,d)$-biregular graph and lossless expansion from large girth only holds when the left-degree is large enough.
\begin{remark}
We can get arbitrary aspect ratios in \Cref{thmtwoside} since edge-vertex incidence graphs can be arbitrarily unbalanced. However, the \cite{lps88} construction used in \Cref{thmlosslessunique} has to be regular, which prevents us from using it directly as base graphs in \cref{thmtwoside}.
\end{remark}
\begin{remark}
Edge-vertex incidence graph has $(2,d)$ bidegree, which means we only have an average degree $2$ from one side. Therefore, our `girth to expansion' reduction \Cref{girthtoexpansion} that works for `all degrees at least $2$' is necessary here. The reduction in \cite{hmmp24} built from spectral analysis only works for degrees at least $3$.
\end{remark}
\subsubsection{Auxilary Tools}\label{sec:tool}
\paragraph{Existence of Biregular Graph of Optimal Two-sided Expansion} \cite{hmmp24} adapts the proof in \cite{randombook} from regular graphs to biregular graph settings and shows the existence of biregular graphs with optimal two-sided expansion. However, they only give the existence statement when degree $d=\Theta(\sqrt{n})$. In our proof \Cref{thmlosslessunique}, we need the case for $d=\Omega(n)$. Motivated by this, we generalize their results to all $\omega(\log{n})\leq d\leq O(n)$. See \Cref{smallbiregular,linearbiregular} for more details.
\paragraph{Subgraph Density Bound from Small Second Largest Eigenvalue} We basically use the same subgraph density bound as in \cite{hmmp24}. However, we slightly relax the parameter ranges in the statement in order to match its application in \Cref{thmtwoside}. Moreover, there is a flaw (will be explained in \Cref{spectralsec}) in the original proof of \cite{hmmp24} and we have to do a slightly different analysis to fix it. Therefore, we reprove it in \Cref{spectralsec} for soundness and completeness.

\subsection{Paper Organization}
In \Cref{prelim} we introduce necessary notations, definitions, and some folklore results. In \Cref{girthsec} we give a new reduction from large girth/bicycle-freeness to unique-neighbor (lossless) expansion. In \Cref{twosidesec} we generalize \cite{lazebnik1,lazebnik2} and give the strongly explicit construction presented in \Cref{thmlaze}. In \Cref{losslesssec,edgevertexsec} we give formal proofs of \Cref{thmlosslessunique,thmtwoside} respectively. Finally, we show the existence of biregular graphs with optimal two-sided expansion in \Cref{biregularsec}, which will be used in our proofs.
\section{Preliminaries}\label{prelim}
\subsection{Graph Notations}
For a graph $G$, we use $G=(V,E)$ to denote its vertex set $V$ and edge set $E\subseteq V\times V$, or equivalently $V(G)$ and $E(G)$. When $G$ is bipartite, we sometimes use $G=(L\cup R,E)$ where $L$ and $R$ distinguish its disjoint left-vertex set and right-vertex set respectively. In this case $E\subseteq L\times R$. A graph is $d$-regular iff all its vertices have degree $d$. A bipartite graph $G=(V\cup R,E)$ is $(c,d)$-biregular iff all its left-vertices have degree $c$ and all its right-vertices have degree $d$. A basic equality in this case is $c|L|=d|R|$. For any vertex $v\in V(G)$, we use $\dgr_G(v)$ to denote its degree in $G$. Fix any $n_1,n_2\in \mathbb{Z}^*,0<p<1$, $\mathbb{G}_{n_1,n_2,p}$ is the Erdős–Rényi bipartite graph distribution that samples a bipartite graph $G=(L\cup R,E)\sim \mathbb{G}_{n_1,n_2,p}$, where $|L|=n_1,|R|=n_2$, and each potential edge $(x,y)\in L\times R$ is contained in $E$ with independent $p$ probability.

For an undirected graph $G=(V,E)$, its girth $g(G)$ is the length of its shortest cycle. For any two of its vertices $u,v\in V\times V$, $d_G(u,v)$ denotes the distance between $u$ and $v$ in $G$. For any vertex set $S\subseteq V$ and $u\in V$, $d_G(u,S)=\min_{s\in S}\{d_G(u,s)\}$ denotes the distance from $u$ to the nearest vertex in $S$. we use $G[S]$ to denote the induced subgraph of $G$ on set $S$. $N_G(S)=\{v\colon v\in V(G),\exists s\in S, (v,s)\in E(G)\}$ denote the set of $S$-neighbors. Similarly, for any positive integer $r>0$, $N^r_G(S)=\{v\colon v\in V(G),d_G(v,S)\leq r\}$ denotes the set of vertices with distance at most $r$ from $S$. For any $v\in N_G(S)$, if $v\notin S$ and $v$ has degree $1$ in $G[S\cup\{v\}]$, we call $v$ a unique-neighbor of $S$. We use $\mathsf{UN}_G(S)$ to denote the set of unique-neighbors of $S$ in $G$. We will sometimes omit the subscript $G$ when it is clear from context.

We next define the notion "$r$-bicycle-free" named from \cite{everydegree}:
\begin{definition}[Bicycle-freeness]\label{bicycle}
For any undirected graph $G$ and $r>0$, $G$ is called $r$-bicycle-free iff for any vertex $v\in V(G)$, let $S^r_v=N^r_G(v)$ denote all vertices with at most $r$ distance from $v$, the induced subgraph $G[S^r_v]$ contains at most $1$ cycle. 
\end{definition}
For an undirected $G$, consider its adjacency matrix $A_G$ and $n=|V(G)|$ eigenvalues of $A_G$ (Counting multiplicities): $\lambda_1\ge \lambda_2\ge\dots\ge\lambda_n$, its second largest eigenvalue (absolute value) $\lambda_2=\lambda_2(G)=\max(|\lambda_2|,\dots,|\lambda_n|)$ will be of particular interest for (bi)regular graphs in this paper.
\begin{definition}
A $d$-regular graph $G$ is a Ramanujan graph iff $\lambda_2(G)\leq 2\sqrt{d-1}$. A $(c,d)$-biregular graph $G$ is a Ramanujan graph iff $\lambda_2(G)\leq \sqrt{c-1}+\sqrt{d-1}$.
\end{definition}
\subsection{Number Theory}
\begin{proposition}[Dirchlet's Theorem]\label{dirchlet}
For any two positive coprime integers $a,d$, there are infinitely many primes of form $a+nd$ where $n\in\mathbb{Z}^+$
\end{proposition}
\begin{proposition}[Prime Number Theorem on Arithmetic Progression \cite{pnt}]\label{pnt}
For any coprime positive integers $a,d$, there exists a constant $C$ such that if we use $p_n$ to denote the $n$-th prime of form $a+md,m\in\mathbb{N}$, then  $\lim_{n\to+\infty}\frac{p_n}{Cn\ln{n}}=1$
\end{proposition}

\section{Expansion from Girth and Bicycle-freeness}\label{girthsec}
In this section, we aim to show \Cref{girthtoexpansion,bicycletoexpansion}, that in any bipartite graph with large girth or bicycle-freeness, any small enough left (right) vertex set with average-degree $d$ should have good vertex-expansion to the other side. Our theorems strengthen the similar arguments of \cite{kahale95,hmmp24} in two facets. First, they yield an additional $g$ factor and give an asymptotically better bound. (where $g$ denotes the girth/bicycle-freeness). Second, our theorems hold for all $d\ge 2$, but the reduction in \cite{hmmp24} only works for $d\ge 3$. This is useful in our proof of \Cref{thmtwoside}, where we use the edge-vertex incidence graph whose left vertices have degree $2$.

Our proofs are from analyzing random induced subgraphs.
\begin{lemma}\label{weakcounting}
Fix any $g$ and $k\leq \lfloor\frac{g}{4}\rfloor$, Let $G=(L\cup R,E)$ be a bipartite graph with $n=|L|+|R|$ vertices, $|E|=m\ge n$ edges, and girth at least $g+1$. then, there is an edge $e\in E$ and $k$ different simple paths containing $e$ that starts from some vertex in $L$, ends at some vertex in $L$, and has length $2k$.   
\end{lemma}
\begin{proof}
Since $m\ge n$, $G$ has some simple cycle $C$ with length at least $g+1$. Consider any edge $e=(u,v)$  on the cycle where $u\in L,v\in R$ and define the direction from $u$ to $v$ as clockwise on $C$. We can find at least $k$ vertices $u_0=u,u_1,\dots,u_{k-1}\in L$ on the cycle that can be reached from $u$ by at most $2k$ clockwise cycle-edges. For each $u_i$, the counter-clockwise cycle-path from $u_i$ with length $2k$ gives a $2k$ simple path containing $e$, and all these paths are distinct. These give us $k$ simple paths with length $2k$ containing $e$ as desired.
\end{proof}
We then give a similar argument for bicycle-free bipartite graphs.
\begin{lemma}\label{weakcountingbicycle}
Fix any $g$ and $k\leq \lfloor\frac{g}{2}\rfloor$, Let $G=(L\cup R,E)$ be a $g$-bicycle-free bipartite graph with $n=|L|+|R|$ vertices and $|E|=m\ge n+1$ edges. then, there is an edge $e\in E$ and $k$ different simple paths containing $e$ that starts from some vertex in $L$, ends at some vertex in $L$, and has length $2k$.   
\end{lemma}
\begin{proof}
Let $C$ denote the number of components in $G$, and the maximal spanning forest of $G$ contains $n-C$ edges. Therefore, fix any maximal spanning forest $F$, there are at least $m-(n-C)\ge C+1$ edges not in the forest. By pigeonhole principle, there exists two of them $e_1,e_2$ in the same component. Each of them defines a cycle in the component, we call them $C_1,C_2$, and let $S=V(C_1)\cup V(C_2)$ denote the set of related vertices. Let $T$ denote the minimal connected subgraph of $F$ that contains $S$. Since $F$ is a forest and vertices in $S$ are from the same component of $F$, $T$ is a well-defined tree. Let $U=V(T)$, we know $G[U]$ has at least two cycles $C_1$ and $C_2$. Suppose $T$ has a diameter at most $2g$, then there exists a vertex $v\in U$ such that all vertices in $U$ have distance at most $g$ from $v$. It implies $U\subseteq N_G^g(v)$ and by above, $G$ is not $g$-bicycle-free, which violates the precondition. Therefore, $T$ has a diameter of at least $2g+1$ and there is a simple path $P$ of length at least $2g+1$ in $G$. By a similar argument as in \Cref{weakcounting}, we can find an edge $e$ and $k$ different simple paths containing $e$ that starts from and ends at $L$ of length $2k$ on $P$
\end{proof}
With \Cref{weakcounting,weakcountingbicycle}, we can build a counting lemma to lower bound the number of long simple paths.
\begin{lemma}[Counting Lemma]\label{medcounting}
Fix any $g$ and $k\leq \lfloor\frac{g}{4}\rfloor$, Let $G=(L\cup R,E)$ be a bipartite graph with $n=|L|+|R|$ vertices, $|E|=m$ edges, and girth at least $g+1$. then, there are at least $k(m-n+1)$ simple paths that starts from some vertex in $L$, ends at some vertex in $L$, and has length at most $2k$. 
\end{lemma}
\begin{proof}
By \Cref{weakcounting}, the statement is true when $m=n$. Now fix any $C\ge 0$ and let's suppose for any $m\leq n+C$ the statement is true, we will show the statement is also true for $m=n+C+1$ and finish the proof by induction.

Suppose $m=n+C+1$, then by \Cref{weakcounting} there is an edge $e$ and $k$ simple paths $p_1,\dots,p_k$ containing $e$ satisfying the conditions. Let $P=\{p_1,\dots,p_k\}$ and delete $e$ from $G$, we will get graph $G'=(L\cup R,E'=E-\{e\})$ where $|E'|=m-1=n+C$. By inductive assumption, $G'$ has at least $k(C+1)$ such paths and none of them is in $P$ since all paths in $P$ contain $e$. Therefore, $G$ has at least $k(C+1)+|P|=k(C+2)$ such paths, we are done.
\end{proof}

By a similar induction, we can use \Cref{weakcountingbicycle} to get a similar counting lemma for $g$-bicycle-free graphs as follows:

\begin{lemma}[Counting Lemma from Bicycle-freeness]\label{medcountingbicycle}
Fix any $g$ and $k\leq \lfloor\frac{g}{2}\rfloor$, Let $G=(L\cup R,E)$ be a $g$-bicycle-free bipartite graph with $n=|L|+|R|$ vertices and $|E|=m$ edges. then, there are at least $k(m-n)$ simple paths that starts from some vertex in $L$, ends at some vertex in $L$, and has length at most $2k$. 
\end{lemma}
With the help of \Cref{medcounting,medcountingbicycle}, we are ready to derive reductions from large girth/bicycle-freeness to expansion:
\begin{theorem}\label{girthtoexpansion}
Fix any $\epsilon\in[0,1],1<d'<\epsilon d$, there exists a constant $\delta=\delta(\epsilon,d,d')>0$ such that the following holds: For any $g>1$, let $G=(L\cup R,E)$ be a bipartite graph whose left average degree is $d$ with girth at least $g+1$, we have:

For any $S\subseteq L$ such that $|S|\leq \delta g d'^{\lfloor g/4\rfloor}$, there is $|N(S)|\ge (1-\epsilon)d|S|$
\end{theorem}
\begin{proof}
Let $k=\lfloor g/4\rfloor$. Suppose by contradiction there is $S\subseteq L$ with $|S|\leq \delta gd'^{\lfloor g/4\rfloor}$ and $|N(S)|\leq (1-\epsilon)d|S|$, where $\delta$ to be determined, let $T=N(S)$ denote the set of $S$-neighbors and $G_S=G[S\cup T]$ be the related induced subgraph. Define $p=\frac{1}{d'}\leq 1$, we will consider a random induced subgraph of $G_S$ sampled as follows: Let $T'\subseteq T$ denote a random subset of vertices in $T$ such that we independently choose each $v\in T$ into $T'$ with $p$ probability, and $H=G_S[S\cup T']$ is the corresponding random induced subgraph of $G_S$. Let $\mathcal{P}_{G_S}$ and $\mathcal{P}_{H}$ denote the set of length-$2k$ simple paths whose starting and ending vertices are both in $L$ in $G_S$ and $H$ respectively, we can get the following bound by using linearity of expectation on each such path in $G_S$:
\begin{align}\label{pathupperbound}
\mathbb{E}_H[|\mathcal{P}_{H}|]\leq \sum_{P\in \mathcal{P}_{G_S}}\mathbb{E}_H[P\subseteq H]\leq |\mathcal{P}_{G_S}|p^k
\end{align}
On the other side, from \Cref{medcounting} we know $|\mathcal{P}_{H}|\ge k(|E(H)-V(H)|)$. Thus we have:
\begin{align}
\mathbb{E}_H[|\mathcal{P}_{H}|]&\ge k(\sum_{e\in E(G_S)}\mathbb{E}_H[e\in E(H)]-|S|-\sum_{t\in T}\mathbb{E}_H[t\in V(H)])\\
&\ge k(|S|dp-|S|-(1-\epsilon)d|S|p)\\
&\ge k|S|(\epsilon dp-1)\label{pathlowerbound}
\end{align}
Combine \Cref{pathupperbound} and \Cref{pathlowerbound}, we can get $|\mathcal{P}_{G_S}|\ge p^{-k}k|S|(\epsilon dp-1)$.
For any unordered pair $(u,v)\in S^2$, if there are two distinct length-$2k$ simple paths $P_1,P_2\in \mathcal{P}_{G_S}$ and both of them start from $u$ and end at $v$, then $P_1\cup P_2$ has to contain a simple cycle of length at most $4k\leq g$, which violates the girth condition. Therefore, all paths in $\mathcal{P}_{G_S}$ have distinct starting/ending vertex pairs. There are only at most $\frac{|S|^2}{2}$ such pairs. Therefore, by pigeonhole principle, we get $|\mathcal{P}_{G_S}|\leq \frac{|S|^2}{2}$. Combine with above lower bound, we get $p^{-k}k|S|(\epsilon dp-1)\leq |\mathcal{P}_{G_S}|\leq \frac{|S|^2}{2}$.

By above, we get $|S|\ge 2(\epsilon dp-1)kp^{-k}>\frac{(\frac{\epsilon d}{d'}-1)}{5}gd'^{\lfloor g/4\rfloor}$. Let $\delta=\delta(\epsilon,d,d')=(\frac{\epsilon d}{d'}-1)/5>0$, this bound violates the precondition $|S|\leq \delta gd'^{\lfloor g/4\rfloor}$. Therefore we must have $|N(S)|>(1-\epsilon)d|S|$.
\end{proof}
By a very similar argument as above combined with \Cref{medcountingbicycle}, we can get a similar expansion from $g$-bicycle-freeness.
\begin{theorem}\label{bicycletoexpansion}
Fix any $\epsilon\in[0,1],1<d'<\epsilon d$, there exists a constant $\delta=\delta(\epsilon,d,d')>0$ such that the following holds: For any $g>1$, let $G=(L\cup R,E)$ be a $g$-bicycle-free bipartite graph whose left average degree is $d$, we have:

For any $S\subseteq L$ such that $|S|\leq \delta g d'^{\lfloor g/2\rfloor}$, there is $|N(S)|\ge (1-\epsilon)d|S|$
\end{theorem}
\begin{proof}
Let $k=\lfloor g/2\rfloor$, and all other notations the same as proof of \Cref{girthtoexpansion}. We use basically the same arguments and replace \Cref{medcounting} with \Cref{medcountingbicycle}. The only difference is the upper bound of $|\mathcal{P}_{G_S}|$. For any unordered pair $(u,v)\in S^2$, if there are three distinct length-$2k$ simple paths $P_1,P_2,P_3\in \mathcal{P}_{G_S}$ and all of them start from $u$ and ends at $v$, then $U=V(P_1)\cup V(P_2)\cup V(P_3)$ must be contained in $N_{G_S}^g(u)$ and thus $G[N_{G_S}^g(u)]$ contains at least two simple cycles by case analysis, which violates $g$-bicycle-freeness. Therefore we get $|\mathcal{P}_{G_S}|\leq |S|^2$. All other parts are similar to the proof above.
\end{proof}
\section{Strongly Explicit Construction of Two-sided Lossless Expander}\label{twosidesec}
In this section, we will give a strongly explicit construction of two-sided lossless expanders with arbitrary aspect ratio $\beta>0$ and arbitrary large bidegree for vertex set with size up to $n^{1/3-\epsilon}$, which matches the expansion derived from the explicit graph family with the largest known girth $\frac{4}{3}\log{n}$

First, we need the biregular bipartite graph $D(k,q)$ defined in \cite{lazebnik1} which has a large girth and good structure.

The \emph{incidence structure} $(P,L,I)$ is a triple where $P$ and $L$ are two disjoint sets (A set of points and a set of lines, respectively), and $I$ is a symmetric binary relation on $P\times L$ called \emph{incidence relation}. Let $B=B(P,L,I)$ be a bipartite graph such that $V(B)=P\cup L$ where $P(L)$ is its left (right) part, and $E(B)$ is defined as $E(B)=\{(p,l)\in I, p\in P,l\in L\}$. By this definition, $B$ is a simple bipartite graph, and we call $B$ the \emph{incidence graph} for the incidence structure $(P,L,I)$.

Let $q$ be a prime power and $k\ge 3$, we define the incidence structure $M(k,q)=(P,L,I)$ as follows:
\begin{align}\label{definitionp}
P&=\{(p_1,p_{1,1},p_{1,2},p_{2,1},p_{2,2},p'_{2,2},p_{2,3},p_{3,2},p_{3,3},p'_{3,3},\dots,p'_{i,i},p_{i,i+1},p_{i+1,i},p_{i+1,i+1},\dots)\in \mathbb{Z}_q^k\}\\\label{definitionl}
L&=\{(l_1,l_{1,1},l_{1,2},l_{2,1},l_{2,2},l'_{2,2},l_{2,3},l_{3,2},l_{3,3},l'_{3,3},\dots,l'_{i,i},l_{i,i+1},l_{i+1,i},l_{i+1,i+1},\dots)\in \mathbb{Z}_q^k\}
\end{align}
\Cref{definitionp,definitionl} mean truncated vectors up to the $k$-th dimension. We assume $p_{-1,0}=l_{0,-1}=p_{1,0}=l_{0,1}=0,p_{0,0}=l_{0,0}=-1,p'_{0,0}=l'_{0,0}=1,p_{0,1}=p_1,l_{1,0}=l_1,l'_{1,1}=l_{1,1},p'_{1,1}=p_{1,1}$, and define incidence relation $L$ as follows: For any $p\in P$ and $l\in L$, $(p,l)\in I$ if and only if the following four equalities hold for each $i=1,2,\dots$. Here, if for an equality, any of its variables is not defined in the $k$-truncated $P,L$ defined above, we just ignore it.
\begin{align}\label{conditions0}
&l_{i,i}-p_{i,i}=l_1p_{i-1,i}\\ \label{conditions1}
&l'_{i,i}-p'_{i,i}=p_1l_{i,i-1}\\ \label{conditions2}
&l_{i,i+1}-p_{i,i+1}=p_1l_{i,i}\\ \label{conditions3}
&l_{i+1,i}-p_{i+1,i}=l_1p'_{i,i}
\end{align}
\begin{remark}
Here we briefly explain why $(P,L,I)$ can be seen as incidence relation between some points and lines. Here, our points are in fact points in $k$-dimensional linear space $\mathbb{F}_q^k$.  for any $p=(p_0,\dots,p_{k-1})$, we can interpret it as the coordinates of a point within the $k$-dimensional space. Then, given any $l=(l_0,\dots,l_{k-1})\in L$ and $p_0$, we think $l_0,\dots,l_{k-1}$ as constants. By observing \Cref{conditions0,conditions1,conditions2,conditions3}, we can do induction on $i\in[k-1]$ to show that for any $p=(p_0,\dots,p_{k-1})$, if  $l$ and $p$ satisfy all these equations, $p_i,i>0$ must be in the form $p_i=a_ip_0+b_i$, where $a_i,b_i$ only depend on $i$ and $l$. Therefore, by fixing $l$, we actually fix these $a_i,b_i$ coefficients, and all `points' $p$ incident to $l$ are exactly those $k$-dimensional points in a specific one-dimensional affine subspace defined by $l$. This means, $l$ is a `line' that defines a one-dimensional subspace. 
\end{remark}
\begin{definition}[\cite{lazebnik1}]
For any prime power $q$ and $k>3$, we define graph $D(k,q)$ to be the bipartite graph $B(M(k,q))$ defined above.
\end{definition}
Fix any $q,k$ satisfying conditions above and incidence structure $M(k,q)=(P,L,I)$. We consider $D(k,q)=(L\cup R,E)$ as bipartite graph, we know $|L|=|R|=|P|=|L|=q^k$. Furthermore, by observing \Cref{conditions0,conditions1,conditions2,conditions3}, we can find that for any $p\in P$ and $t\in \mathbb{Z}_q$ there is exactly one $l\in L$ such that $(p,l)\in I$ and $l_1=t$. That's because when we want to find solutions $l\in L$ for the above equalities, after fixing $p$ and any $l_1$, all other indices of $l$ can be iteratively calculated. There must be exactly one such solution $l$.

Therefore, for any $p\in P$ there are exactly $q$ solutions $l^{(0)},\dots,l^{(q-1)}\in L$ such that $(p,l^{(i)})\in I$ for all $0\leq i<q$, where $l^{(i)}_1=i$ so these solutions are pairwise distinct. This implies $D(k,q)$ is $q$-left regular bipartite graph. A similar analysis also holds for $L$ and the right part. It follows that $D(k,q)$ is actually a $(q,q)$-biregular bipartite graph.

This construction is very useful for its large girth:
\begin{theorem}[\cite{lazebnik1}]\label{girthlaze}
For any prime power $q$ and $k>3$, $D(k,q)$ has girth at least $k+4$.
\end{theorem}
We can further reduce the number of vertices while preserving the girth unchanged by identifying a single connected component of $D(k,q)$. We have the following result for the structure of connected components of $D(k,q)$.

To characterize connected components of $D(k,q)$, we need the following definition for \emph{certificate}. For any 
prime power $q$, odd $k\ge 6$ and corresponding incidence structure $M(k,q)=(P,L,I)$, given any $u=(u_1,u_{1,1},u_{1,2},u_{2,1},u_{2,2},u'_{2,2},u_{2,3},u_{3,2},u_{3,3},u'_{3,3},\dots,u'_{i,i},u_{i,i+1},u_{i+1,i},u_{i+1,i+1},\dots)\in P\cup L$ and any $r\leq \lfloor\frac{k+2}{4}\rfloor$, we define $r$-\emph{certificate} $a_r(u)$ of $u$ as follows:
\begin{align}\label{defcerti}
\forall 2\leq t\leq r, b_t(u)&=\sum_{i=0}^t(u_{i,i}u'_{r-i,r-i}-u_{i,i+1}u_{r-i,r-i-1})\\
a_r(u)&=(b_2(u),\dots,b_r(u))
\end{align}
Here we use $p_{0,-1}=l_{0,-1}=p_{1,0}=l_{0,1}=0,p_{0,0}=l_{0,0}=-1,p'_{0,0}=l'_{0,0}=1,p_{0,1}=p_1.l_{1,0}=l_1,l'_{1,1}=l_{1,1},p'_{1,1}=p_{1,1}$.

It's easy to show that given the above constraints on odd $k\ge 6$ and $r\leq \lfloor\frac{k+2}{4}\rfloor$, all related variables in \Cref{defcerti} are defined and thus the $r$-certificate $a_r(u)$ is well defined for any $u\in P\cup L$. Then, we can give a characterization of connected components in $D(k,q)$.
\begin{theorem}[\cite{lazebnik2}]\label{component}
Given prime power $q$ and odd $k\ge 6$, let $r=\lfloor\frac{k+2}{4}\rfloor$. Consider $M(k,q)=(P,L,I)$ and  $D(k,q)=(P\cup L,I)$, then for any two vertices $u,v\in (P\cup L)$ in the same connected component of $D(k,q)$, we have $a_r(u)=a_r(v)$.
\end{theorem}
From above, we can use a single connected component of $D(k,q)$ instead of the whole while preserving the girth unchanged.
\begin{definition}\label{defcd}
Given prime power $q$ and odd $k\ge 6$, let $r=\lfloor\frac{k+2}{4}\rfloor$. Consider $D(k,q)=(P\cup L,I)$ defined above, let $S=\{u\colon a_r(u)=0^{r-1},u\in P\cup L\}$ denote the set of vertices with all-zero $r$-certificate vector. We define the graph $CD(k,q)=D(k,q)[S]$ as the induced subgraph of $D(k,q)$ on $S$.
\end{definition}
From the definition above it's obvious that $0^k\in S$ so $CD(k,q)$ is non-empty. Moreover, by \Cref{component} we know $CD(k,q)$ consists some complete connected component in $D(k,q)$, which means $CD(k,q)$ is a $(q,q)$-biregular graph.

The remaining problem is that this graph now is still balanced. To make it unbalanced with arbitrary ratio, we need to consider its specific induced subgraph as follows:
\begin{definition}
Given prime power $q$ and odd $k\ge 6$, consider $CD(k,q)=(P[S]\cup L[S],I[S])$ defined above. Let $A,B\subseteq \mathbb{Z}^*_q$ be two non-empty subsets, and $U=\{p\colon p_1\in A,p\in P[S]\}, V=\{l\colon l_1\in B,l\in L[S]\}$ be the sets of left/right vertices whose first indices are within $A$ and $B$ respectively, we define $CD(k,q,A,B)=CD(k,q)[U\cup V]$ as the induced subgraph of $CD(k,q)$ over left/right vertices with specified first index sets $A,B$.
\end{definition}
From the previous observation, we know for each $p\in U$ and $b\in B$, there is exactly one vertex $l\in V$ such that $(p,l)\in E(CD(k,q))$. Therefore, $CD(k,q,A,B)=(U\cup V,E)$ is a $(|B|,|A|)$-biregular graph with girth at least $k+4$. Thus, we can get arbitrary unbalance factor $\beta=\frac{|A|}{|B|}$ by choosing $A,B$ with appropriate sizes. Next, to get a strongly explicit construction guarantee, we need to characterize vertices in $CD(k,q,A,B)$. We will use the same notations as in \Cref{defcd} for brevity.

For any $u=(u_1,u_{1,1},u_{1,2},u_{2,1},u_{2,2},u'_{2,2},u_{2,3},u_{3,2},u_{3,3},u'_{3,3},\dots,u'_{i,i},u_{i,i+1},u_{i+1,i},u_{i+1,i+1},\dots)\in P\cup L$, we can define four corresponding polynomials $A_u,B_u,C_u,D_u\in \mathbb{Z}_q[x]/(x^{r+1})$ modulo $x^{r+1}$ as follows:
\begin{align}
A_u(x)=\sum_{i=0}^ru_{i,i}x^i,B_u(x)=\sum_{i=0}^ru'_{i,i}x^i,C_u(x)=\sum_{i=0}^ru_{i,i+1}x^i,D_u(x)=\sum_{i=0}^ru_{i,i-1}x^i

\end{align}
For simplicity, we use $a_i,b_i,c_i,d_i$ to denote the $i$-th coefficient of $A_u(x),B_u(x),C_u(x),D_u(x)$. From the above definitions, we can derive the conditions when $u\in U\cup V$:

1. If $u\in P$, then $u\in U$ iff:
\begin{align}\label{cddef1}
d_0=d_1=0;a_0=-1,b_0=1,c_0\in A,a_1=b_1\\\label{cddef2}
A_u(x)B_u(x)-C_u(x)D_u(x)\equiv t_1+t_2x\mod{x^{r+1}},t_1,t_2\in\mathbb{Z}_q
\end{align}

2. If $u\in L$, then $u\in V$ iff:
\begin{align*}
d_0=c_0=0;a_0=-1,b_0=1,d_1\in B,a_1=b_1\\
A_u(x)B_u(x)-C_u(x)D_u(x)\equiv t_1+t_2x\mod{x^{r+1}},t_1,t_2\in\mathbb{Z}_q
\end{align*}
\begin{theorem}\label{stronglyconstruction}
Let $CD(k,q,A,B)=(U\cup V,E)$ be the $(|B|,|A|)$-biregular graph defined above. We can construct it strongly explicitly, and $|U|=|A|q^{k+1-r},|V|=|B|q^{k+1-r}$.  
\end{theorem}
\begin{proof}
We first prove $|U|=|A|q^{k+1-r}$, and then $|V|=|B|q^{k+1-r}$ follows from a similar argument:

It's not hard to count that there are exactly $|A|q^{3r-1}$ solutions for $A_u(x),B_u(x),C_u(x)$ that satisfying \Cref{cddef1}. Fix any of them, if there exists some $D_u(x)$ with $d_0=d_1=0$ and \Cref{cddef2} holds, then there must be $t_1=-1$ and $t_2=0$ in \Cref{cddef2}. Then, note that $c_0\in A$ and $A\subseteq \mathbb{Z}^*_q$, so $C_u(x)$ must be invertible in $\mathbb{Z}_q[x]/(x^{r+1})$ and thus has an inverse $C^{-1}_u(x)$. Therefore, we can directly solve \Cref{cddef2} and get the desired solution $D_u(x)$ as $D_u(x)\equiv (A_u(x)B_u(x)+1)C^{-1}_u(x)\mod{x^{r+1}}$. Note that $A_u(x)B_u(x)+1\equiv 0\mod{x^2}$, this solution must satisfy $d_0=d_1=0$. Therefore, for each of $|A|q^{3r-1}$ solutions $A_u(x),B_u(x),C_u(x)$ of \Cref{cddef1} there is exactly one $D_u(x)$ that satisfies \Cref{cddef1} and \Cref{cddef2}. Fix these four polynomials, there are $k-4r+2$ remaining free indices in $u\in \mathbb{Z}_q^k$ unrelated to the definition of these polynomials, which means there are $q^{k-4r+2}$ ways to assign them. These assignments map to distinct $u\in U$, and conversely each $u\in U$ defines such a unique assignment. Therefore, by bijection counting we get $|U|=|A|q^{3r-1}\times q^{k-4r+2}=|A|q^{k+1-r}$.

We can easily define an arbitrary lexicographical order on solutions $A_u(x),B_u(x),C_u(x)$ of \Cref{cddef1} and assignments to free variables. Let $n=|U|$, we can use the order to define a function $f_U\colon [n]\rightarrow \mathbb{Z}_q^k$ such that $f_U(i)$ gives the vector encoding $\widetilde{u}_i\in\mathbb{Z}_q^k$ of the vertex $u_i$ with the $i$-th smallest vector in $U$. By basic polynomial calculation described above, we can evaluate $f_U$ and $f^{-1}_U$ within $\poly(r)=\polylog(|U|)$ time. A similar argument also holds for the right-side and we define the corresponding functions $f_V,f^{-1}_V$.

Therefore, given $i\leq n$ and $j\leq |B|$,  in order to compute the $j$-th neighbor of $u_i\in U$, we first compute its vector encoding $\widetilde{u}_i=f_U(i)$, set $l_1=b_j$(Here $b_j$ is the $j$-th element of $B$), and use \Cref{conditions0,conditions1,conditions2,conditions3} to get $\widetilde{v}\in \mathbb{Z}_q^k$, which is the vector encoding of its $j$-th neighbor. Finally, compute $f_V^{-1}(\widetilde{v})$ and we can get the index of $j$-th neighbor of $u_i$ in $V$. The total time complexity is $\polylog(|U|+|V|)$
\end{proof}
Now we are ready to construct two-sided lossless expanders with arbitrary unbalanced ratios. 
\begin{theorem}\label{thmlaze}
Let $0\leq \beta_0<\beta_1\leq 1,\epsilon_1,\epsilon_2\in(0,1]$. There are infinitely many pairs of $(d_1,d_2),d_1\leq d_2$ and $\delta=\delta(d_1,d_2,\epsilon_1,\epsilon_2)$ where $\beta=\frac{d_1}{d_2}\in [\beta_0,\beta_1]$, such that there is a strongly explicit construction of an infinite family of $(d_1,d_2)$-biregular graphs $(G_n=(L_n\cup R_n,E_n))_n$ where:

1. For all $S\subseteq L_n$ with $|S|\leq \delta|L_n|^{1/3-\epsilon_2}$, $N(S)\ge (1-\epsilon_1)d_1|S|$.

2. For all $S\subseteq R_n$ with $|S|\leq \delta|R_n|^{1/3-\epsilon_2}$, $N(S)\ge (1-\epsilon_1)d_2|S|$.
\end{theorem}
\begin{proof}
There are infinitely many prime powers $q$ such that we can choose $\beta_0(q-1)\leq d_1\leq \beta_1(q-1)$. Let $d_2=q-1$, these pairs $(d_1,d_2)$ satisfy the specified unbalance condition. 

Let $A=[1,d_1],B=[1,d_2]$, by \Cref{stronglyconstruction}, we have a strongly explicit construction of $CD(k,q,A,B)$ for all sufficiently large odd $k\ge 6$. Let $G_m=CD(2m+101,q,A,B)$ be the graph we consider. We can suppose $q,d_1,d_2$ are some sufficiently large enough constants. 

We first prove the case when $S\subseteq L_m$. Let $k=2m+101, r=\lfloor\frac{k+4}{4}\rfloor,n=|V(G_m)|\leq q^{k+3-r}$ by \Cref{stronglyconstruction}. Because 
$G_m$ has girth at least $k+4$ by \Cref{girthlaze}, from \Cref{girthtoexpansion} we know that when $|S|\leq \delta'(\epsilon_1(d_1-2))^{k/4}$, $N(S)\ge (1-\epsilon_1)d_1|S|$, where $\delta'=\delta'(\epsilon_1,d_1)$ (We can assume $\epsilon_1(d_1-2)>1$ by using large enough $d_1$). By $n\leq q^{k+3-r}\leq q^{3k/4+4}$, we can get $k\ge \frac{4}{3}(\log_q{n}-4)$. Therefore, we can rewrite the bound as follows:
\begin{align*}
&\delta'(\epsilon_1(d_1-2))^{k/4}\\
\ge& \delta'(n/q^4)^{1/3(\log_q(\epsilon_1(d_1-2)))}\\
\ge&\delta'(n/q^4)^{1/3(\log_q{\epsilon_1}+\log_q{0.1\beta_0}+1)}
\end{align*}
Let $a=\log_q{\epsilon_1}+\log_q{0.1\beta_0}$. 
By setting $q$ large enough we can make $a\ge-\epsilon_2$. Let $\delta=\delta'/q^{4/3(1+a)}$, the above bound is larger than $\delta n^{1/3-\epsilon_2}$, and we are done. The lossless expansion from right to left follows from a similar argument.
\end{proof}
\section{Lossless Expander with Additional Expansion from the Other Side}\label{losslesssec}
In this section, we aim to give a strongly explicit construction of lossless expanders that not only yields lossless expansion from left to right for linear-sized sets but also guarantees all $n^{1/3-\epsilon}$-sized sets from the right side have a constant fraction of neighbors to be unique neighbors.

We use the tripartite product introduced in \cite{hmmp24} as the framework (will be formally defined later). This product needs two large base graphs and a constant-size gadget graph. We first prepare the two base graphs we will need:

\begin{theorem}[\cite{lps88}]\label{lps88cons}
We have a strongly explicit construction algorithm $\mathcal{A}$ that: Let $p,q$ be two primes where $p\equiv q\equiv 1\mod{4}$ where $q$ is not a quadratic residue modulo $p$, $\mathcal{A}(p,q)$ construct a $(p+1,p+1)$-biregular graph graph $G=(L\cup R,E)$ where $|L|=|R|=q(q^2-1)/2$ such that
\begin{compactitem}
\item[(1)] $\lambda_2(G)\leq 2\sqrt{p}$

\item[(2)] girth of $G$ is at least $4\log_p{q}$.
\end{compactitem}
\end{theorem}
We can use \Cref{lps88cons} to derive the first base Ramanujan graph we need with large girth.
\begin{theorem}\label{lps88cons2}
For any prime $p\equiv 1\mod 4$, we have a strongly explicit construction of an infinite family of $(p+1,p+1)$-biregular graphs $(G_n)_{n\ge 1}$ such that $\lambda_2(G_n)\leq 2\sqrt{p}$, and $G_n$ has girth at least $\frac{4}{3}\log_p{|V(G_n)|}$.
\end{theorem}
\begin{proof}
By \Cref{dirchlet} there are infinitely many odd primes $q\equiv p+(3p+1)g\mod{4p}$ where $g$ is the primitive root of $\mathbb{Z}_p^*$. By Chinese remainder theorem we know such $q$ satisfies $q\equiv 1\mod{4}$ and $q\equiv g\mod{p}$, and therefore $q$ is not a square residue modulo $p$. Using \cref{lps88cons}, we get the desired construction.
\end{proof}
Another base graph we use is from \cite{ow20}.
\begin{theorem}[A special case of main theorem in \cite{ow20}]\label{ow20cons}
For every $c,d\ge 3$ and $\gamma>0$, there is a constant $\delta>0$ and a strongly explicit construction of an infinite family of $(c,d)$-biregular graphs $(G_n)_{n\ge 1}$ such that $\lambda_2(G_n)\leq (\sqrt{c-1}+\sqrt{d-1})(1+\gamma)$, and $G_n$ is $(\epsilon\sqrt{\log{|V(G_n)|}})$-bicycle-free.
\end{theorem}
Then, in order to construct the desired gadget graph by exhaustive search, we will need the existence of 
$(d_1,d_2)$-biregular graphs with optimal two-sided expansion in the settings $d_1=\Theta(n)$.
\begin{theorem}[\Cref{linearbiregular2} Restated]\label{linearbiregular}
For any $\delta,C>0$, there exists a constant $\epsilon>0$ such that: Let $n_1=n\leq n_2\leq \poly(n),\omega(\log{n_1})\leq d_1\leq \epsilon n_2,\omega(\log{n_2})\leq d_2\leq \epsilon n_1$ and $p=\frac{d_1}{n_2}=\frac{d_2}{n_1}$. For all sufficiently large $n$, there exists $(d_1,d_2)$-biregular graph with $n_1$ and $n_2$ left/right vertices such that:
\begin{compactitem}
\item[(1)] For all $t\leq \frac{Cn_1}{d_2}$,  $\forall S\subseteq R(H)$, s.t. $|S|=t$, $\frac{|\mathsf{UN}_H(S)|}{|S|}\ge (1-\delta)d_2\exp(-pt)$

\item[(2)] For all $t\leq \frac{Cn_2}{d_1}$,  $\forall S\subseteq L(H)$, s.t. $|S|=t$, $\frac{|\mathsf{UN}_H(S)|}{|S|}\ge (1-\delta)d_1\exp(-pt)$
\end{compactitem}
\end{theorem}
The proof of \Cref{linearbiregular} is postponed to \Cref{biregularsec}.

The last tool we will use is the following theorem that uses the second largest eigenvalue to bound induced subgraph density in $(c,d)$-biregular graph. This is generalized from the similar result in \cite{hmmp24}. We fix a flaw in the proof of \cite{hmmp24}, and also generalize the parameters a bit from $\min\{c,d\}\ge 3$ to $\min\{c,d\}\ge 2$, which will be useful in \Cref{edgevertexsec}.
\begin{theorem}\label{subgraphexpansion}
Let $2\leq c\leq d$ be integers that $cd>6$ and $0<\epsilon<0.01$. For any $(c,d)$-biregular graph $G=(L\cup R,E)$ and its vertex set $S\subseteq L\cup R$ that $|S|\leq d^{-1/\epsilon}|L\cup R|$, let $L_S=S\cap L,R_S=S\cap R, m=|E(G[S])|$ and $d_L=m/|L_S|,d_R=m/|L_R|$ denote the left/right average degree in $G[S]$, we have:
\begin{equation*}
(d_L-1)(d_R-1)\leq \lambda^2-(\sqrt{c-1}-\sqrt{d-1})^2
\end{equation*}
where
$
\lambda=\max(\lambda_2(G),\sqrt{c-1}+\sqrt{d-1})(1+5\epsilon)$
\end{theorem}
We defer the proof of \Cref{subgraphexpansion} to \Cref{spectralsec}.

Then, we give a formal definition of Tripartite Product, and prove \Cref{thmlosslessunique}.
\begin{definition}[Tripartite Product]\label{tripartite}
Given a $d_1$-right-regular bipartite graph $G_1=(L\cup M,E_1)$, a $d_2$-left-regular bipartite graph $G_2=(M\cup R,E_2)$ and a bipartite graph $G_0=(L'\cup R',E_0)$ where $|L'|=d_1$ and $|R'
|=d_2$, their tripartite product is defined as a bipartite multigraph $G=(L\cup R, E)$ where:
\begin{equation*}
E=\{\!\!\{(n_{G_1}(w,i),n_{G_2}(w,j))\colon \forall w\in M,\forall (i,j)\in E_0\}\!\!\}.
\end{equation*}
Here $n_{G_t}(w,k)$ means the $k$-th neighbor of $w$ in graph $G_t$.
\end{definition}
\begin{theorem}\label{thmlosslessunique}
For any $\epsilon>0$, there exists $\delta,C>0$ such that: For infinitely many $d_1,d_2$ where $d_1\leq Cd_2$, there exists 
$\delta'$ and a strongly explicit construction of an infinite family of $(d_1,d_2)$-biregular graphs $(G_n=(L_n\cup R_n,E_n))_{n\ge 1}$ such that:
\begin{compactitem}
\item[(1)] For any $S\subseteq L_n$ where $|S|\leq \delta' |L_n|$, we have $|N(S)|\ge (1-\epsilon)d_1|S|$.

\item[(2)] For any $S\subseteq R_n$ where $|S|\leq |R_n|^{1/3-\epsilon}$, $|\mathsf{UN}_G(S)|\ge \delta d_2|S|$.

\item[(3)] For any $S\subseteq R_n$ where $|S|\leq \exp(\delta'\sqrt{\log{|R_n|}})$, we have $|N(S)|\ge (1-\epsilon)d_2|S|$
\end{compactitem}
\end{theorem}
\begin{proof}
By \Cref{dirchlet}, there are infinitely many odd primes $p\equiv 1\mod{4}$. For any sufficiently large such $p$. By \Cref{lps88cons2}, we can construct an infinite family of graph $G_1=(L\cup M,E_1)$ such that:
\begin{compactitem}
\item[(1)] $G_1$ is a $(p+1,p+1)$-biregular graph.

\item[(2)] Girth of $G_1$ is at least $\frac{4}{3}\log_p{|L|}$

\item[(3)] $G_1$ is a Ramanujan graph, where $\lambda_2(G_1)\leq 2\sqrt{p}$.
\end{compactitem}
Let $C_2=C_2(\epsilon)$ be some large enough constant to be determined later, and define $D=p+1, D'=D/C$ for some large enough $C=C(\epsilon,C_2)$ to be determined later. By \Cref{ow20cons}, we can construct an infinite family of $(D,D')$-biregular graph $G_2=(M\cup R,E_2)$ that:
\begin{compactitem}
\item[(1)] $G_2$ is a near-Ramanujan graph, where $\lambda_2(G_2)\leq (\sqrt{D-1}+\sqrt{D'-1})(1+1/(C_2D))$.

\item[(2)] There exists $\epsilon'=\epsilon'(\epsilon,C,C_2,D)$ such that $G_2$ is $(\epsilon'\sqrt{\log{|R|}})$-bicycle-free. 
\end{compactitem}
Then we need to construct the gadget graph $G_0=(L'\cup R',E_0)$. Setting $\delta_S=1/C_2,C_S=31C_2^2,d_S=C_2\sqrt{(D-1)(D'-1)}$ be parameters in \Cref{linearbiregular} and invoke it, we can construct a bipartite graph $G_0=(L'\cup R',E_0)$ by exhaustive search such that:
\begin{compactitem}
\item[(1)] $|L'|=|R'|=D$, $G_0$ is a $(d,d)$-biregualr graph for some large enough $C:=C(\epsilon,C_2)$ where $d=d_S$.

\item[(2)] For all  $S\subseteq L(H)$  (or $R(H)$) where $|S|\leq \frac{31C_2^2D}{d}$, we have $\frac{|\mathsf{UN}_{G_0}(S)|}{|S|}\ge (1-1/C_2)d\exp(-|S|d/D)$.
\end{compactitem}
Our final graph should be tripartite product (\Cref{tripartite}) $G=(L\cup R,E)$ of $G_1,G_2$ and $G_0$. We know $G$ is a $(d_1=Dd,d_2=D'd)$-biregular graph with $d_1\leq Cd_2$. Now we are ready to prove the three expansion properties:
\paragraph{Left-to-right Lossless Expansion for Linear-sized Sets}

Setting $\epsilon_S=10^{-5}/(C_2D)$ be the parameter in \Cref{subgraphexpansion} and invoke it, it follows that 
there exists $\delta_1>0$ such that:

1. For any $S_L\subseteq L$ and $S_R\subseteq M$ with $|S|\leq \delta_1D^3|L\cup M|$, where $S=S_L\cup S_R$, let $d_L,d_R$ denote the left/right average degree of $G_1[S]$, we have:
\begin{equation}\label{subgraphexpg1}
(d_L-1)(d_R-1)\leq 21(D-1)
\end{equation}

2. For any $S_L\subseteq M$ and $S_R\subseteq R$ with $|S|\leq \delta_1D^3|M\cup R|$, where $S=S_L\cup S_R$, let $d_L,d_R$ denote the left/right average degree of $G_2[S]$, we have:
\begin{equation}\label{subgraphexpg2}
(d_L-1)(d_R-1)\leq 25\sqrt{(D'-1)(D-1)}
\end{equation}

Now fix any $S\subseteq L$ where $|S|\leq \delta_1|L|$, we want to guarantee its unique-neighbor expansion. From the definition of $G_0$, we know $d=C_2\sqrt{(D-1)(D'-1)}$. Define $h:=42C_2+1$, and we let $\mathcal{U}=N_{G_1}(S)$ denote the set of neighbors of $S$ in $G_1$ and $\mathcal{U}_h\subseteq \mathcal{U}$ denote the set of neighbors with degree at least $h$ in $G_{1,S}:=G_1[S\cup \mathcal{U}]$. Let $d_h$ denote the average degree of $S$ in $G_1[S\cup \mathcal{U}_h]$, by \Cref{subgraphexpg1}, we know:
\begin{equation}\label{fewhighdegree}
d_h\leq 1+\frac{D}{2C_2}
\end{equation}
It follows that $\frac{D}{C_2d}\ge \frac{\sqrt{C}}{C_2^2}\ge 42C_2+1=h$ by setting $C:=C(\epsilon,C_2)$ large enough. Therefore, 
By setting $C$ large enough, it follows that $h\leq \frac{D}{C_2d}$.

For any $u\in \mathcal{U}_l=\mathcal{U}\backslash\mathcal{U}_h$, let $S_u$ be the set of $u$'s neighbors in $G_{1,S}$ and thus $|S_u|\leq h$ is its degree in $G_{1,S}$. Consider the $G_0$ copy $G^u_0$ placed on $u$ as gadget graph, we define $T_u\subseteq R$ as the set of unique-neighbors of $S_u$ in $G^u_0$.  By property of $G_0$, we can lower bound $|T_u|$ as:
\begin{equation}\label{lbgadget}
|T_u|\ge (1-1/C_2)|S_u|d\exp(-|S_u|d/D)\ge (1-1/C_2)|S_u|d\exp(-1/C_2)
\end{equation}
Let $T:=(\bigcup_{u\in\mathcal{U}_l}T_u)$ denote these vertices. It follows that for any $v\in T$, if $v$ has degree $1$ in $G':=G_2[\mathcal{U},T]$, then $v$ must be a unique-neighbor of $S$ in $G$. Let $d'_L,d'_R$ be the left/right average degree of graph $G'$, we first lower bound $d'_L$: First, by \Cref{fewhighdegree}, we can get:
\begin{equation}\label{lbu}
|\mathcal{U}|=|\mathcal{U}_l|+|\mathcal{U}_h|\leq |\mathcal{U}_l|+|S|d_h/h\leq |\mathcal{U}_l|+\frac{|\mathcal{U}_l|d_h}{(D-d_h)}\leq(1+1/C_2)|\mathcal{U}_l|
\end{equation}

By \Cref{lbu} we can lower bound $d'_L$ as:
\begin{equation}\label{lbdl}
d'_L\ge\frac{|\mathcal{U}_l|d}{|\mathcal{U}|}\ge \frac{d}{(1+1/C_2)}
\end{equation}

By \Cref{subgraphexpg2} on $G'$ and \Cref{lbdl}, we can finally upper bound the right average degree $d'_R$ as:
\begin{equation}\label{ubdr}
d'_R\leq 1+\frac{25\sqrt{(D-1)(D'-1)}}{d'_L-1}\leq 1+30/C_2
\end{equation}

It follows from \Cref{ubdr,fewhighdegree,lbgadget} that:
\begin{align*}
|N_G(S)|&\ge|T|\ge \frac{\sum_{u\in\mathcal{U}_l}|T_u|}{d'_R}\ge\frac{(1-1/C_2)\exp(-1/C_2)d\sum_{u\in\mathcal{U}_l}|S_u|}{1+30/C_2}\\
&\ge\frac{(1-1/C_2)\exp(-1/C_2)d|S|(D-d_h)}{1+30/C_2}\\
&\ge \frac{(1-1/C_2)^2\exp(-1/C_2)}{1+30/C_2}Dd|S|
\end{align*}
By setting $C_2=C_2(\epsilon)$ large enough, by the above lower bound we can get:
\begin{equation*}
N(S)\ge (1-\epsilon)Dd|S|=(1-\epsilon)d_1|S|
\end{equation*}
\paragraph{Right-to-left Unique-neighbor Expansion for Polynomial-size Sets}
Let $S\subseteq R, |S|\leq |R|^{1/3-\epsilon}$ be any subset of right vertices with polynomial-bounded size. We want to show its unique-neighbor expansion. Define $\mathcal{U}\subseteq M$ as the set of neighbors of $S$ in $G_2$. Let $h=30C_2\sqrt{C}+1$, we use the similar definitions as in the previous paragraph that $\mathcal{U}_h$ is the set of all vertices with an average degree at least $h$ in $G_{2,S}=G_2[\mathcal{U}\cup S]$. Define  $\mathcal{U}_l=\mathcal{U}\backslash\mathcal{U}_h$, and $d_h$ as the average degree of $S$ in $G_2[\mathcal{U}_h\cup S]$. By \Cref{subgraphexpg2}, we have:
\begin{equation}\label{fewhighdegree2}
d_h\leq 1+\frac{25\sqrt{(D-1)(D'-1)}}{h-1}\leq 1+\frac{5D}{6C_2C}\leq \frac{6D}{7C_2C}\leq \frac{6D'}{7C_2}
\end{equation}

For any $u\in \mathcal{U}_l$,  let $S_u$ be the set of $u$'s neighbors in $G_{2,S}$ and thus $|S_u|\leq h$ is its degree in $G_{2,S}$. Consider the $G_0$ copy $G^u_0$ placed on $u$ as gadget graph, we define $T_u\subseteq L$ as the set of unique-neighbors of $S_u$ in $G^u_0$.  By property of $G_0$, we can lower bound $|T_u|$ as:
\begin{equation}\label{lbgadgetinm}
|T_u|\ge (1-1/C_2)|S_u|d\exp(-|S_u|d/D)\ge(1-1/C_2)|S_u|d\exp(-|S_u|C_2/\sqrt{C})
\end{equation}

Consider $\delta_3=\delta_3(\epsilon,C_2,C)=\min_{1\leq t\leq h}\{(1-1/C_2)t\exp(-tC_2/\sqrt{C})\}$ as a constant depending on $C,C_2$. By \Cref{lbgadgetinm}, we get that:
\begin{equation}\label{lbgadgetr}
|T_u|\ge(1-1/C_2)|S_u|d\exp(-|S_u|C_2/\sqrt{C})\ge \delta_3 d
\end{equation}

Let $T:=(\bigcup_{u\in\mathcal{U}_l}T_u)$ denote these vertices. It follows that for any $v\in T$, if $v$ has degree $1$ in $G':=G_1[T\cup \mathcal{U}]$, then $v$ must be a unique-neighbor of $S$ in $G$. Our goal is to establish a lower bound on the number of such vertices $v$. Similarly to the previous argument, we first need an upper bound on $|\mathcal{U}|$. The following is from \Cref{fewhighdegree2}:
\begin{equation}\label{lbu2}
|\mathcal{U}|=|\mathcal{U}_l|+|\mathcal{U}_h|\leq |\mathcal{U}_l|+|S|d_h/h\leq |\mathcal{U}_l|+\frac{|\mathcal{U}_l|d_h}{(D'-d_h)}\leq(1+1/C_2)|\mathcal{U}_l|
\end{equation}

The lower bound of right average degree $d'_R$ of $G'$ follows from \Cref{lbu2}:

\begin{equation}\label{lbdr2}
d_R'\ge\frac{|\mathcal{U}_l|\delta_3 d}{|\mathcal{U}|}\ge\frac{\delta_3 d}{1+1/C_2}
\end{equation}

When $d=\frac{C_2D}{\sqrt{C}}$ is large enough, we can guarantee $d_R'/4C_2\ge \frac{\delta_3 d}{4C_2+4}>1$. Setting $\epsilon_S=1/(2C_2), d'_S=d'_R/(4C_2)$ as parameters in the statement of \Cref{girthtoexpansion} and invoke it, we know that: Let $\delta_S:=\delta_S(C_2)$ be some constant and $g\ge \frac{4}{3}\log_{D-1}|M|$ is the girth of $G_1$, we can guarantee: If we have $|\mathcal{U}|\leq \delta_S(d'_R/(4C_2))^{g/4-1}$ (We call this \emph{Condition A}), then there are at least $(1-\epsilon_S)d'_R|\mathcal{U}|$ neighbors in $G'$, which implies $G'$ has at least $(1-1/C_2)d'_R|\mathcal{U}|$ unique-neighbors of $\mathcal{|U|}$. By the previous argument, it automatically implies $|\mathsf{UN}_G(S)|\ge (1-1/C_2)d'_R|\mathcal{U}|$.

By some calculation, we can upper bound $|\mathcal{U}|$ as:
\begin{align*}
|\mathcal{U}|&\leq |S|D'\leq D'|R|^{1/3-\epsilon}\leq (C|M|)^{1/3-\epsilon/2}&\text{(}|M| \text{ large enough and }|R|=C|M|\text{)}\\
&\leq \delta_S(\frac{|M|^{1/3}}{D-1})^{1-\epsilon}&\text{(}|M| \text{ large enough)}\\
&\leq \delta_S(\frac{|M|^{1/3}}{D-1})^{(\log_{D-1}D+\log_{D-1}(\frac{\delta_3}{4C(1+1/C_2)}))}&\text{(}D\text{ large enough)}\\
&= \delta_S(\frac{|M|^{1/3}}{D-1})^{(\log_{D-1}(\frac{\delta_3C_2D}{C(1+1/C_2)})-\log_{D-1}(4C_2))}&\\
&=\delta_S(\frac{|M|^{1/3}}{D-1})^{(\log_{D-1}(\frac{\delta_3 d}{1+1/C_2})-\log_{D-1}(4C_2))}&\\
&\leq \delta_S(\frac{|M|^{1/3}}{D-1})^{\log_{D-1}(\frac{d'_R}{4C_2})}&\text{(\Cref{lbdr2})}\\
&=\delta_S(\frac{d'_R}{4C_2})^{\frac{1}{3}\log_{D-1}(|M|)-1}&\\
&\leq \delta_S(d'_R/(4C_2))^{g/4-1}&
\end{align*}
Therefore, we know \emph{Condition A} holds, and from the previous discussion, we know that:
\begin{align*}
|\mathsf{UN}_G(S)|&\ge (1-1/C_2)d'_R|\mathcal{U}|\ge (1-1/C_2)\sum_{u\in\mathcal{U}_l}|T_u|&\\
&\ge (1-1/C_2)^2\sum_{u\in\mathcal{U}_l}|S_u|d\exp(-|S_u|C_2/\sqrt{C})&\text{(\Cref{lbgadgetinm})}\\
&\ge (1-1/C_2)^2\exp(-31C^2_2)d\sum_{u\in\mathcal{U}_l}|S_u|&\text{(}|S_u|\leq h\text{)}\\
&\ge (1-1/C_2)^2\exp(-31C^2_2)d(D'-d_h)|S|&\\
&\ge (1-1/C_2)^3\exp(-31C^2_2)d_2|S|&\text{(\Cref{fewhighdegree2})}
\end{align*}
Setting $\delta=\delta(C_2,C)=(1-1/C_2)^3\exp(-31C)$, from above we can guarantee $|\mathsf{UN}_G(S)|\ge\delta d_2|S|$, which proves the desired statement.
\paragraph{Right-to-left Lossless Expansion for Subpolynomial-size Sets}
Let $\delta_4:=\epsilon'/10$. For any $S\subseteq R$ that $|S|\leq \exp(\delta_4\sqrt{\log{|R|}})$, our goal is to prove its lossless expansion.

Setting $d_S=D',\epsilon_S=1/(2C_2)$ and $d'_S=D'/(4C_2)$ as parameters in the statement of \Cref{bicycletoexpansion} and invoke it, we can guarantee that for some $\delta_S$, if $|S|\leq \delta_S(\frac{D'}{4C_2})^{g/2-1}$, then $|N_{G_2}(S)|\ge (1-1/(2C_2))D'|S|$, where $g\ge \epsilon\sqrt{\log{|R|}}$ is the bicycle-freeness of $G_2$. We can verify that $|S|$ satisfies this condition:
\begin{equation*}
|S|\leq \exp(\delta_4\sqrt{\log{|R|}})\leq \delta_S(\frac{D'}{4C_2})^{(1.1\delta_4\sqrt{\log{|R|}})/(\log(\frac{D'}{4C_2}))}\leq \delta_S(\frac{D'}{4C_2})^{0.11\epsilon'\sqrt{\log|R|}}\leq \delta_S(\frac{D'}{4C_2})^{g/2-1}
\end{equation*}
Therefore, $|S|$ has at least $(1-1/(2C_2))D'|S|$ neighbors in $G_2$, which implies it has at least 
 $|\mathsf{UN}_{G_2}(S)|\ge (1-1/C_2)D'|S|$ unique-neighbors. Let $\mathcal{U}=\mathsf{UN}_{G_2}(S)$, for each $u\in\mathcal{U}$, consider the gadget graph $G^u_0$ placed on $u$. Since $u$ is a unique-neighbor of $S$ in $G_2$, there are at least $d$ vertices in $L$ that are also neighbors of $S$ in $G$. Let $T_u$ denote the set of these vertices and $T:=\bigcup_{u\in\mathcal{U}}T_u\subseteq N_G(S)$ be a subset of $N_G(S)$, from the above discussion we know that for each $u\in\mathcal{U}$, it has degree at least $d$ in $G':=G_1[T\cup \mathcal{U}]$. Let $d'_R\ge d$ denote the real right average degree of $G'$. 

 Setting $\epsilon_S=1/C_2, d'_S=d'_R/(2C_2)$ as parameters in statement of \Cref{girthtoexpansion} and invoke it, we know that: Let $\delta_S:=\delta_S(C_2)$ be some constant and $g\ge \frac{4}{3}\log_{D-1}|M|$ be the girth of $G_1$, If we have $|\mathcal{U}|\leq \delta_S(d'_R/(2C_2))^{g/4-1}$ (We call this \emph{Condition B}), then there are at least $(1-1/C_2)d'_R|\mathcal{U}|\ge (1-1/C_2)d|\mathcal{U}|$ neighbors of $\mathcal{U}$ in $G'$. This also implies $|T|\ge (1-1/C_2)^2dD'|S|=(1-1/C_2)^2d_2|S|$. Since $T\subseteq N_{G}(S)$, we can finish the proof just by setting $C_2$ large enough such that $(1-1/C_2)^2\ge 1-\epsilon$ given \emph{Condition B}.
 
\emph{Condition B} can be easily verified as:
\begin{align*}
|\mathcal{U}|&\leq D'\exp(\delta_4\sqrt{\log|R|})\leq \delta_S(\frac{d'_R}{2C_2})^{1.1\delta_4\sqrt{\log(D-1)\log_{D-1}(|M|)}}\\
&\leq \delta_S(\frac{d'_R}{2C_2})^{0.0001\log_{D-1}|M|}\leq \delta_S(d'_R/(2C_2))^{g/4-1}.
\end{align*}
The above holds since $\frac{d'_R}{2C_2}\ge\frac{D}{2\sqrt{C}}>e$ by large enough $D$, and $|M|$ is also large enough.

\end{proof}
\section{Two-sided Unique-neighbor Expanders with Better Expansion for Polynomial-sized Sets}\label{edgevertexsec}
In this section, we give strongly explicit construction of two-sided unique-neighbor expanders for $\Theta(n)$-sized sets. At the same time, we guarantee that for a small polynomial $p$, all vertex sets with size at most $p(n)$ have almost `half of' neighbors to be unique neighbors. We use the same tripartite product framework. However, this time we use two edge-vertex incidence graphs as base graphs.
\begin{definition}\label{edgevertexdef}
Given any $d$ regular graph $G=(V,E)$, its edge-vertex incidence graph $G'=(E\cup V, E')$ is defined as the $(2,d)$-biregular bipartite graph where the left vertex set is $E$, the right vertex set is $V$, and for each left vertex $e=(u,v)\in E$, it has exactly two neighbors $u,v\in V$ in $G'$ 
\end{definition}
We need the arguments below to derive the spectral bound of edge-vertex incidence graph.
\begin{definition}[Irreducible Matrix]
For any square matrix $M\in \mathbb{R}^{n\times n}$, let $G_M=([n],E)$ be the \emph{directed} graph where for any $i,j\in[n],(i,j)\in E$ iff $M_{ij}\neq 0$. If $G_M$ is a strongly connected graph, we call $M$ an irreducible matrix.  
\end{definition}
\begin{fact}\label{connected}
For any undirected graph $G$, if $G$ is connected, then its adjacency matrix $A_G$ is an irreducible matrix.
\end{fact}
\begin{theorem}[Perron-Frobenius Theorem \cite{minc}]\label{pfthm} For any irreducible non-negative symmetric square matrix $M\in\mathbb{R}^{n\times n}$ whose spectral radius is $r=\rho(M)$, $\lambda=r$ is its largest eigenvalue with algebraic multiplicity $1$.
\end{theorem}
\begin{proposition}[\cite{spectral}]\label{maxspectral}
Any $d$-regular graph $G$ has spectral radius $\rho(A_G)=d$, where $A_G$ is its adjacency matrix.
\end{proposition}
Now we are ready to state a correspondence between eigenvalues of regular graph $G$ and its edge-vertex incidence graph $G'$:
\begin{lemma}\label{eigencorrespond}
For any $d$-regular graph $G=(V,E)$ and its real eigenvalue $\lambda\neq -d$, its corresponding edge-vertex incidence graph $G'=(E\cup V,E')$ has two eigenvalues $\lambda'=\pm\sqrt{\lambda+d}$.

Conversely, for any non-zero eigenvalue $\lambda'\neq 0$ of $G'$, $G$ has eigenvalue $\lambda=\lambda'^2-d$.
\end{lemma}
\begin{proof}
For any eigenvalue $\lambda\neq-d$, let 
$\lambda'=\pm\sqrt{\lambda+d}\neq 0$ and $x\in\mathbb{C}^n$ denote its corresponding eigenvector where $A_Gx=\lambda x$, we can build a vector $x'\in\mathbb{C}^{n+m}$ such that ($n,m$ denotes the number of vertices and edges in $G$ respectively):
\begin{equation*}
\forall u\in E\cup V, x'_u=
\begin{cases}
x_u,&u\in V\\
\frac{x_a+x_b}{\lambda'},&u=(a,b)\in E
\end{cases}
\end{equation*}
Now we want to verify that $A_{G'}x'=\lambda'x'$ and thus $\lambda'$ is an eigenvalue of $G'$:

For any $u=(a,b)\in E$, we have:
\begin{equation*}
(A_{G'}x')_u=x'_a+x'_b=x_a+x_b=\lambda' x'_u
\end{equation*}

For any $u\in V$, we have:
\begin{equation*}
(A_{G'}x')_u=\sum_{e=(u,v)\in E}x'_e=\frac{1}{\lambda'}\sum_{e=(u,v)\in E}(x'_u+x'_v)=\frac{d}{\lambda'}x'_u+\frac{(A_Gx)_u}{\lambda'}=\frac{d+\lambda}{\lambda'}x'_u=\lambda'x'_u
\end{equation*}
This whole argument is reversible, so both directions hold.
\end{proof}
\begin{theorem}\label{spectraledgevertex}
For any connected $d$-regular graph $G$ with the second largest eigenvalue $\lambda_2(G)$, its corresponding edge-vertex incidence graph $G'$ has the second largest eigenvalue $\lambda_2(G')=\sqrt{\lambda_2(G)+d}$
\end{theorem}
\begin{proof}
First, by \Cref{maxspectral} we know $G$ has two eigenvalues $d$ and $\lambda_2=\lambda_2(G)\leq d$, and by \Cref{eigencorrespond} we know $G'$ has two eigenvalues $\sqrt{2d}$ and $\lambda_2'=\sqrt{\lambda_2+d}$, therefore, $\lambda_2(G')\ge \lambda'_2$.

It suffices to prove $\lambda_2(G')\leq \lambda'_2$ to complete the proof. Suppose by contradiction that $G'$ has two eigenvalues $\lambda'_a\ge\lambda'_b>\lambda'_2$. Since $G$ and $G'$ is connected,  we know from \Cref{connected} that $A_{G'}$ is irreducible and therefore by \Cref{pfthm} we can assume $\lambda'_a=\rho(A_{G'})>\lambda'_b>\lambda'_2\ge 0$. Then, by \Cref{eigencorrespond} again $G$ has three eigenvalues $\lambda_a=\lambda'^2_a-d,\lambda_b=\lambda'^2_b-d$ and $\lambda_2=\lambda'^2_2-d=\lambda_2(G)$, where $\lambda_a>\lambda_b>\lambda_2(G)$, which contradicts the definition of $\lambda_2(G)$. Therefore, it follows that $\lambda_2(G')\leq \lambda'_2$
\end{proof}
\begin{proposition}[\cite{lpsbook}]\label{lpsconnected}
When $q>p^8$, the construction in \Cref{lps88cons} yields a connected graph.
\end{proposition}
Next, we need the biregular graph with optimal two-sided expansion as gadget graph. The difference from that in \Cref{losslesssec} is here we need the degree to be $\Theta(\sqrt{n})$:
\begin{theorem}[\Cref{smallbiregular2}, Restated]\label{smallbiregular}
Let $C>0,n_1=n\leq n_2\leq \poly(n),\omega(\log{n_1})\leq d_1\leq o(n_2),\omega(\log{n_2})\leq d_2\leq o(n_1)$ and $p=\frac{d_1}{n_2}=\frac{d_2}{n_1}$. For all sufficiently large $n$, there exists $(d_1,d_2)$-biregular graph with $n_1$ and $n_2$ left/right vertices such that:
\begin{compactitem}
\item[(1)] For all $t\leq \frac{Cn_1}{d_2}$,  $\forall S\subseteq R(H)$, s.t. $|S|=t$, $\frac{|\mathsf{UN}_H(S)|}{|S|}\ge (1-o(1))d_2\exp(-pt)$

\item[(2)] For all $t\leq \frac{Cn_2}{d_1}$,  $\forall S\subseteq L(H)$, s.t. $|S|=t$, $\frac{|\mathsf{UN}_H(S)|}{|S|}\ge (1-o(1))d_1\exp(-pt)$
\end{compactitem}
\end{theorem}
Now we are ready to prove the main theorem in this section \Cref{thmtwoside}.
\begin{theorem}\label{thmtwoside}
For any $\epsilon>0,0\leq \beta_0<\beta_1\leq 1$, there exists $\delta=\delta(\beta_0,\beta_1)$ and infinitely many pairs $(d_1,d_2)$ where $\frac{d_1}{d_2}\in[\beta_0,\beta_1]$ such that: there exists 
$\delta'$ and a strongly explicit construction of an infinite family of $(d_1,d_2)$-biregular graphs $(G_n=(L_n\cup R_n,E_n))_n$ such that:
\begin{compactitem}
\item[(1)] For any $S\subseteq L_n$ where $|S|\leq \delta' |L_n|$, we have $|\mathsf{UN}_{G_n}(S)|\ge \delta d_1|S|$.

\item[(2)] For any $S\subseteq R_n$ where $|S|\leq \delta' |R_n|$, we have $|\mathsf{UN}_{G_n}(S)|\ge \delta d_2|S|$.

\item[(3)] For any $S\subseteq L_n$ where $|S|\leq |L_n|^{\delta'}$, we have $|\mathsf{UN}_{G_n}(S)|\ge (\frac{1}{2}-\epsilon)d_1|S|$.

\item[(4)] For any $S\subseteq R_n$ where $|S|\leq |R_n|^{\delta'}$, we have $|\mathsf{UN}_{G_n}(S)|\ge (\frac{1}{2}-\epsilon)d_2|S|$.
\end{compactitem}
\end{theorem}
\begin{proof}
Let $\beta=\frac{\beta_0+\beta_1}{2}$, and $p_n\equiv 1\mod{4}$ be the $n$-th prime of form $4k+1$. From \Cref{pnt}, we know $\lim_{n\to+\infty}\frac{p_{\beta n}}{p_n}=\lim_{n\to+\infty}\beta(1+\frac{\ln{\beta}}{\ln{n}})=\beta$. Therefore, for all sufficiently large $n$, we know $\frac{p_{\beta n}+1}{p_n+1}\in[\beta_0,\beta_1]$. Let $p_a=p_n,p_b=p_{\beta n},D_1=p_a+1$, $D_2=p_b+1$, and $g_a,g_b$ be the primitive roots of $\mathbb{Z}^*_{p_a},\mathbb{Z}^*_{p_b}$ respectively. We can also define $p'_a$ as inverse of $p_a$ in $\mathbb{Z}^*_{p_b}$ (and similar definition for $p'_b$). Let $M=p_ap_b+p_ap'_a(3p_b+1)g_b+p_bp'_b(3p_a+1)g_a$ and $N=4p_ap_b$. It follows that $N$ and $M$ are coprime. Therefore, from \Cref{dirchlet}, there are infinitely many primes $q$ of form $kN+M,k\in\mathbb{N}$. By Chinese Remainder Theorem, it holds that $q\equiv 1\mod{4},q\equiv g_a\mod{p_a}$ and $q\equiv g_b\mod{p_b}$, which implies $q$ is a quadratic residue modulo neither $p_a$ nor $p_b$. By \Cref{lps88cons}, we can construct an infinite family of graphs $G_a=(L\cup R,E'_a),G_b=(L'\cup R',E'_b)$ such that:
\begin{compactitem}
\item[(1)] $G_a$ is a $(D_1,D_1)$-biregular graph such that $|L|=|R|=\frac{q(q^2-1)}{2}$.

\item[(2)] Girth of $G_a$ is at least $\frac{4}{3}\log_{p_a}|L|$.

\item[(3)] $G_a$ is a Ramanujan graph, where $\lambda_2(G_a)\leq 2\sqrt{p_a}$.
\end{compactitem}
The above three properties also hold similarly for $G_b$. We also have $|L|=|L'|=|R|=|R'|$.

Next, let $H_a=((L=E'_a)\cup (M=L'+R'),E_a)$ and $H_b=((M=L'+R')\cup (R=E'_b),E_b)$ denote the edge-vertex incidence graphs of $G_a$ and $G_b$ respectively. We have the following facts:
\begin{compactitem}
\item[(1)] $H_a$ is a $(2,D_1)$-biregular graph, and $H_b$ is $(D_2,2)$-biregular. $|M|=q(q^2-1)$.

\item[(2)] Both $H_a$ and $H_b$ have girth at least $\frac{8}{3}\log_{p_a}|L'|-1$

\item[(3)] By \Cref{lpsconnected} and \Cref{spectraledgevertex}, $H_a$ and $H_b$ have the second largest eigenvalues upper bounded by $\lambda_2(H_a)\leq \sqrt{2(D_1-1)^{1/2}+D_1},\lambda_2(H_b)\leq \sqrt{2(D_2-1)^{1/2}+D_2}$
\end{compactitem}
Let $C_0=10000, C=C_0^2/\beta_0, d_1=C_0\sqrt{D_2-1},d_2=\frac{d_1D_1}{D_2}$, it follows that $d_1\ge C_0\sqrt{D_2-1}$ and $d_2\ge C_0\sqrt{D_1-1}$.  Then we need to construct the gadget graph $G_0=(L_0\cup R_0,E_0)$ by \Cref{smallbiregular}, our graph has the following properties:
\begin{compactitem}
\item[(1)] $|L_0|=D_1,|R_0|=D_2$, $G_0$ is a $(d_1,d_2)$-biregular graph.

\item[(2)] For all  $S\subseteq L(H)$ where $|S|\leq \frac{CD_2}{d_1}$, we have $\frac{\mathsf{UN}_{G_0}(S)}{|S|}\ge (1-o_{D_1}(1))d_1\exp(-|S|d_1/D_2)$. A similar property also holds for right vertex sets.
\end{compactitem}
Our final graph should be tripartite product (\Cref{tripartite}) $G=(L\cup R,E)$ of $H_a,H_b$ and $G_0$. We know $G$ is a $(2d_1,2d_2)$-biregular graph, which satisfies the unbalanced condition since $\frac{d_1}{d_2}=\frac{D_2}{D_1}\in[\beta_0,\beta_1]$. Now we are ready to prove the expansion properties:
\paragraph{Two-sided Unique-neighbor Expansion for Linear-size Sets}

Setting $\epsilon_S=10^{-5}/(C_0D_1)$ be the parameter in \Cref{subgraphexpansion} and invoke it, it follows that 
there exists $\delta_1>0$ such that:

1. For any $S_L\subseteq L$ and $S_R\subseteq M$ with $|S|\leq \delta_1D_1^3|L\cup M|$, where $S=S_L\cup S_R$, let $d_L,d_R$ denote the left/right average degrees of $H_a[S]$, we have:
\begin{equation}\label{twosubgraphexpg1}
(d_L-1)(d_R-1)\leq 25\sqrt{D_1-1}
\end{equation}

2. For any $S_L\subseteq M$ and $S_R\subseteq R$ with $|S|\leq \delta_1D_1^3|M\cup R|$, where $S=S_L\cup S_R$, let $d_L,d_R$ denote the left/right average degrees of $H_b[S]$, we have:
\begin{equation}\label{twosubgraphexpg2}
(d_L-1)(d_R-1)\leq 25\sqrt{D_2-1}
\end{equation}

Without loss of generality, we just prove unique-neighbor expansion from left to right for linear-sized sets. Now fix any $S\subseteq L$ where $|S|\leq \delta_1|L|$, we want to guarantee its unique-neighbor expansion. Define $h=C_0\sqrt{D_1-1}$. We let $\mathcal{U}=N_{H_a}(S)$ denote the set of neighbors of $S$ in $H_a$ and $\mathcal{U}_h\subseteq \mathcal{U}$ denote the set of neighbors with degree at least $h$ in $H_{1,S}:=H_a[S\cup \mathcal{U}]$. Let $d_h$ denote the average degree of $S$ in $H_a[S\cup \mathcal{U}_h]$. by \Cref{twosubgraphexpg1}, we know:
\begin{equation}\label{twofewhighdegree}
d_h\leq 1+\frac{25}{C_0}
\end{equation}
We know $\frac{D_2}{d_1}\ge \frac{\sqrt{\beta D_1}}{C_0}$, it follows that $\frac{hd_1}{D_2}\leq C_0^2/\beta\leq C$.

For any $u\in \mathcal{U}_l=\mathcal{U}\backslash\mathcal{U}_h$,  let $S_u$ be the set of $u$'s neighbors in $H_{1,S}$ and thus $|S_u|\leq h$ be its degree in $H_{1,S}$. Consider the $G_0$ copy $G^u_0$ placed on $u$ as gadget graph, we define $T_u\subseteq R$ as the set of unique-neighbors of $S_u$ in $G^u_0$.  By property of $G_0$, we can lower bound $|T_u|$ as:
\begin{equation}\label{twolbgadgetinm}
|T_u|\ge (1-o_{D_1}(1))|S_u|d_1\exp(-|S_u|d_1/D_2)
\end{equation}

Consider $\delta_3=\delta_3(\beta,C_0)=\min_{1\leq t\leq h}\{(1-o_{D_1}(1))t\exp(-td_1/D_2)\}$. in fact, it follows that
\begin{align*}
\delta_3&=(1-o_{D_1}(1))\min(\exp(-d_1/D_2),h\exp(-hd_1/D_2))\\
&\ge(1-o_{D_1}(1))\min(\exp(-o_D(1)),C_0\sqrt{D_1-1}\exp(-C))\\
&\ge(1-o_{D_1}(1))
\end{align*}
By \Cref{twolbgadgetinm}, we get that:
\begin{equation}\label{twolbgadgetr}
|T_u|\ge(1-o_{D_1}(1))|S_u|d_1\exp(-|S_u|d_1/D_2)\ge \delta_3 d_1\ge(1-o_D(1))d_1
\end{equation}

Let $T=(\bigcup_{u\in\mathcal{U}_l}T_u)$ denote these vertices. It follows that for any $v\in T$, if $v$ has degree $1$ in $H'=H_b[\mathcal{U}\cup T]$, then $v$ must be a unique-neighbor of $S$ in $G$. Our goal is to establish a lower bound on the number of such vertices $v$. Similarly to the previous argument, we first need an upper bound on $|\mathcal{U}|$. The following is from \Cref{twofewhighdegree}:
\begin{equation}\label{twolbu2}
|\mathcal{U}|=|\mathcal{U}_l|+|\mathcal{U}_h|\leq |\mathcal{U}_l|+|S|d_h/h\leq |\mathcal{U}_l|+\frac{|\mathcal{U}_l|d_h}{(2-d_h)}\leq(1+\frac{1+25/C_0}{1-25/C_0})|\mathcal{U}_l|\leq 3|\mathcal{U}_l|
\end{equation}

The lower bound for left average degree $d'_L$ of $H'$ follows from \Cref{twolbu2}:

\begin{equation}\label{twolbdl}
d_L'\ge\frac{|\mathcal{U}_l|\delta_3 d_1}{|\mathcal{U}|}\ge\frac{(1-o_D(1)) d_1}{3}\ge d_1/4+1
\end{equation}

By \Cref{twosubgraphexpg2} on $H'$ and \Cref{twolbdl}, we can finally upper bound the right average degree $d'_R$ as:
\begin{equation}\label{twoubdr}
d'_R\leq 1+\frac{25\sqrt{D_2-1}}{d'_L-1}\leq 1+100/C_0
\end{equation}

It follows from \Cref{twoubdr,twofewhighdegree,twolbgadgetinm} that:
\begin{align*}
|\mathsf{UN}_G(S)|&\ge(2-d_R')|T|\ge \frac{(2-d_R')\sum_{u\in\mathcal{U}_l}|T_u|}{d'_R}\\
&\ge\frac{(1-100/C_0)(1-o_D(1))\exp(-C)d_1\sum_{u\in\mathcal{U}_l}|S_u|}{1+100/C_0}\\
&\ge0.98\exp(-C)d_1|S|(2-d_h)\\
&\ge 0.48\exp(-C)(2d_1)|S|
\end{align*}
By setting $\delta=0.48\exp(-C)$, the above lower bound tells us the desired unique-neighbor expansion:
\begin{equation*}
|\mathsf{UN}_G(S)|\ge \delta 2d_1|S|
\end{equation*}
\paragraph{Two-sided $(1/2-\epsilon)$ Unique-neighbor Expansion of Polynomial-sized Sets}
Without loss of generality, we just prove 
$(1/2-\epsilon)$ unique-neighbor expansion from left to right. Let $\delta_4>0$ denote some small enough constant to be determined. For any $S\subseteq |L|$ that $|S|\leq |L|^{\delta_4}$, our goal is to prove $|\mathsf{UN}_G(S)|\ge (1/2-\epsilon)2d_1|S|$.

Setting $d_S=2,\epsilon_S=(1/2+\epsilon/200)$ and $d'_S=(1+\epsilon/200)<2\epsilon_S$ as parameters in the statement of \Cref{girthtoexpansion} and invoke it, we can guarantee that for some $\delta_S$, if $|S|\leq \delta_S(d'_S)^{g/4-1}$, then $|N_{H_a}(S)|\ge (1-\epsilon/100)|S|$, where $g\ge \frac{8}{3}\log_{D_1}|L|-4$ is the girth of $H_a$. We can verify that $|S|$ satisfies this condition by setting $\delta_4=\delta_4(\epsilon,D_1)$ small enough:
\begin{equation*}
|S|\leq |L|^{\delta_4}\leq \delta_S(d'_S)^{1.1\delta_4\log_{d'_S}|L|}\leq \delta_S(d'_S)^{1.1\delta_4\log_{d'_S}(D_1)\log_{D_1}|L|}\leq \delta_S(d'_S)^{g/4-1}
\end{equation*}
Therefore, $|S|$ has at least $(1-\epsilon/100)|S|$ neighbors in $H_a$,

Let $h:=200/\epsilon$. We let $\mathcal{U}=N_{H_a}(S)$ denote the set of neighbors of $S$ in $H_a$ and $\mathcal{U}_h\subseteq \mathcal{U}$ denote the set of neighbors with degree at least $h$ in $H_{1,S}:=H_a[S\cup \mathcal{U}]$. Let $d_h$ denote the average degree of $S$ in $H_a[S\cup \mathcal{U}_h]$, we can bound it as follows
\begin{align}
&d_h|S|/h+(2-d_h)|S|\ge|\mathcal{U}|\ge (1-\epsilon/100)|S|\\\label{finalfewhighdegree}
\Rightarrow&d_h\leq \frac{1+\epsilon/100}{1-1/h}\leq 1+\epsilon/50
\end{align}

For any $u\in \mathcal{U}_l=\mathcal{U}\backslash\mathcal{U}_h$, let $S_u$ be the set of $u$'s neighbors in $H_{1,S}$ and $|S_u|\leq h$ is its degree in $H_{1,S}$. Consider the $G_0$ copy $G^u_0$ placed on $u$ as gadget graph, we define $T_u\subseteq R$ as the set of unique-neighbors of $S_u$ in $G^u_0$.  By property of $G_0$, we can lower bound $|T_u|$ as:
\begin{align}
|T_u|&\ge (1-o_{D_1}(1))|S_u|d_1\exp(-|S_u|d_1/D_2)\ge (1-o_{D_1}(1))\exp(-o_{D_1}(1))|S_u|d_1\\\label{finallbgadget}&\ge(1-o_{D_1}(1))|S_u|d_1
\end{align}
Let $T:=(\bigcup_{u\in\mathcal{U}_l}T_u)$ denote these vertices. It follows that for any $v\in T$, if $v$ has degree $1$ in $H':=H_b[\mathcal{U},T]$, then $v$ must be a unique-neighbor of $S$ in $G$. Let $d'_L$ be the left average degree of graph $H'$, we first lower bound $d'_L$: By \Cref{finalfewhighdegree}, we can get:
\begin{equation}\label{finallbu}
|\mathcal{U}|=|\mathcal{U}_l|+|\mathcal{U}_h|\leq |\mathcal{U}_l|+|S|d_h/h\leq |\mathcal{U}_l|+\frac{|\mathcal{U}_l|d_h}{(2-d_h)}\leq(2+\epsilon/10)|\mathcal{U}_l|
\end{equation}

By \Cref{finallbu}, we can lower bound $d'_L$ as:
\begin{equation}\label{finallbdl}
d'_L\ge\frac{|\mathcal{U}_l|d_1}{|\mathcal{U}|}\ge \frac{d_1}{(2+\epsilon/10)}
\end{equation}

 Setting $\epsilon_S=\epsilon/30,d_S=d_L', d'_S=\epsilon d_S/15>1$ as parameters in the statement of \Cref{girthtoexpansion} and invoke it, we know that: Let $\delta_S$ be some constant and $g\ge \frac{8}{3}\log_{D_1}|L|-4$ be the girth of $H_b$ (and also $H'$), if we have $|\mathcal{U}|\leq \delta_S(d'_S)^{g/4-1}$ (We call this \emph{Condition A}), then there are at least $(1-\epsilon_S)d'_L|\mathcal{U}|$ neighbors of $\mathcal{U}$ in $H'$, which implies there are at least $(1-\epsilon/15)d'_L|\mathcal{U}|$ unique-neighbors of $S$ in $G$. it follows from \Cref{finallbgadget,finalfewhighdegree} that:
 \begin{align*}
 |\mathsf{UN}_G(S)|&\ge (1-\epsilon/15)d'_L|\mathcal{U}|\ge (1-\epsilon/15)\sum_{u\mathcal{U}_l}|T_u|\\
 &\ge (1-\epsilon/15)(1-o_{D_1}(1))d_1\sum_{u\in\mathcal{U}_l}|S_u|
 \ge(1-\epsilon/10)d_1(2-d_h)|S|\\
 &\ge(1-\epsilon/10)^2d_1|S|\ge (1/2-\epsilon)2d_1|S|
 \end{align*}
 Therefore, we can finish the proof given \emph{Condition A}.
 
\emph{Condition A} can be easily verified as:
\begin{equation*}
|\mathcal{U}|\leq 2|S|\leq \delta_S|L|^{1.1\delta_4}\leq \delta_S(d'_S)^{1.1\delta_4\log_{d'_S}(D_1)\log_{D_1}|L|}\leq \delta_S(d'_S)^{g/4-1}
\end{equation*}
The above holds by setting $\delta_4=\delta_4(\epsilon,D_1)$ small enough, and $d'_S\ge \frac{\epsilon d_1}{50}\ge \Omega_{\epsilon}(\sqrt{D_1})\ge 2$ for large enough $D_1$ by \Cref{finallbdl}.

Let $\delta'=\delta'(\beta_0,\beta_1,\epsilon,D_1):=\min(\delta_1,\delta_4)$, we complete the whole proof. 
\end{proof}
\section{Random Biregular Graphs}\label{biregularsec}
In this section, we prove \Cref{smallbiregular,linearbiregular}. The method is generalized from \cite{hmmp24,randombook} to biregular graphs and all bidegrees $\omega(\log{n})\leq d\leq O(n)$.
\begin{theorem}[Chernoff Bound]\label{chernoff}
Suppose $X_1,\dots,X_n$ are i.i.d random variables sampled from $\{0,1\}$. Let $X=\sum_{i=1}^nX_i$ and $\mu:=\mathbb{E}[X]$, then for any $\delta>0$, we have:
\begin{equation*}
\mathsf{Pr}[X\leq (1-\delta)\mu]\leq \exp(-\delta^2\mu/2)
\end{equation*}
\end{theorem}
First, we recall the definition of Erdős–Rényi graph sampler:
\begin{definition}
For any $p\in[0,1],n_1,n_2>0$, $\mathbb{G}_{n_1,n_2,p}$ denote the distribution of random bipartite graph $G=(L\cup R,E)$, where $|L|=n_1,|R|=n_2$, and for each $(u,v)\in L\times R$, the edge $(u,v)$ is contained in $E(G)$ with probability $p$ independently.
\end{definition}
\begin{theorem}[Graph Distribution Embedding \cite{hmmp24}]\label{embedding}There exists a constant $C$ such that: 
Fix any $n_1,n_2,d_1,d_2\in \mathbb{N}$ such that $m=n_1d_1=n_2d_2$. Then for $p=(1-C(\frac{d_1d_2}{m}+\frac{\log{m}}{\min\{d_1,d_2\}})^{1/3})\frac{m}{n_1n_2}$, there is a joint distribution $\mathsf{D}=(\mathsf{G},\mathsf{H})$ such that:

1. The marginal distribution of $\mathsf{G}$ is $\mathbb{G}_{n_1,n_2,p}$, the marginal distribution of $\mathsf{H}$ is a uniformly random 
$(d_1,d_2)$-biregular bipartite graph $H=(L\cup R,E)$ where $|L|=n_1,|R|=n_2$.

2. $\mathsf{Pr}_{G,H\sim \mathsf{D}}[E(G)\subseteq E(H)]\ge 1-o(1)$.
\end{theorem}

\begin{lemma}\label{randomsample}
Let $H\sim\mathbb{G}_{n_1,n_2,p}$, where $n=\mathsf{min}\{n_1,n_2\}$ then with probability $1-O(\frac{1}{n})$, for all $t\ge 1$, we have:

1. $\forall S\subseteq R(H)$, s.t. $|S|=t$, $\frac{|\mathsf{UN}_H(S)|}{|S|}\ge p(1-p)^{t-1}n_1-\sqrt{4p(1-p)^{t-1}n_1\log{n_2}}$

2. $\forall S\subseteq L(H)$, s.t. $|S|=t$, $\frac{|\mathsf{UN}_H(S)|}{|S|}\ge p(1-p)^{t-1}n_2-\sqrt{4p(1-p)^{t-1}n_2\log{n_1}}$
\end{lemma}
\begin{proof}
Without loss of generality, we just give the proof for $S\subset R(H)$, $|S|=t$.
\begin{equation*}
|\mathsf{UN}_H(S)|=\sum_{v\in L(H)}\mathsf{1}[v\in \mathsf{UN}_H(S)]
\end{equation*}
For each of $v\in L(H)$, the number of edges $e_v$ between $v$ and $S$ can be seen as the sum of $t$ i.i.d Bernoulli random variables with probability $p$. Let $X_v$ be the indicator function of event $e_v=1$, we know $X_v$ is a Bernoulli random variable of probability $q=tp(1-p)^{t-1}$ and $|\mathsf{UN}_H(S)|=\sum_{v\in L(H)}X_v$. These $\{X_v\}_{v\in L(H)}$ are i.i.d random variables, and $\mathbb{E}[\sum_{v\in L(H)}X_v]=qn_1$. By Chernoff Bound, we have, for all $s\ge 0$:
\begin{equation*}
\mathsf{Pr}[|\mathsf{UN}_H(S)|\leq (1-\frac{s}{\sqrt{qn_1}})qn_1]\leq \exp(-s^2/2)
\end{equation*}

Insert $s=2\sqrt{t\log{n_2}}$ to above we know that with probability at least $1-n_2^{-2t}$, we have 

\begin{equation*}
|\mathsf{UN}_H(S)|\ge tp(1-p)^{t-1}n_1-t\sqrt{4p(1-p)^{t-1}n_1\log{n_2}}
\end{equation*}

By union bound over all $t$, the failure probability is at most $\sum_{t=1}^{+\infty}n_2^tn_2^{-2t}\leq O(\frac{1}{n})$.
\end{proof}
\begin{theorem}\label{smallbiregular2}
Let $C>0,n_1=n\leq n_2\leq \poly(n),\omega(\log{n_1})\leq d_1\leq o(n_2),\omega(\log{n_2})\leq d_2\leq o(n_1)$ and $p=\frac{d_1}{n_2}=\frac{d_2}{n_1}$. For all sufficiently large $n$, there exists $(d_1,d_2)$-biregular graph with $n_1$ and $n_2$ left/right vertices such that:

1. For all $t\leq \frac{Cn_1}{d_2}$,  $\forall S\subseteq R(H)$, s.t. $|S|=t$, $\frac{|\mathsf{UN}_H(S)|}{|S|}\ge (1-o(1))d_2\exp(-pt)$

2. For all $t\leq \frac{Cn_2}{d_1}$,  $\forall S\subseteq L(H)$, s.t. $|S|=t$, $\frac{|\mathsf{UN}_H(S)|}{|S|}\ge (1-o(1))d_1\exp(-pt)$
\end{theorem}
\begin{proof}
By \Cref{embedding}, there exists a distribution $\mathsf{D}$ of $(G,H)$, whose marginal distribution on $G$ is $\mathbb{G}_{n_1,n_2,p'}$ and $H=(L\cup R,E)$ is  always some $(d_1,d_2)$ biregular bipartite graph with $|L|=n_1,|R|=n_2$. Moreover, with probability $1-o(1)$, we have $G\subseteq H$. Here $p'=(1-o(1))p$.

Fix some function $\omega((\log{n}/p'n_2)^{1/2})\leq f(n)\leq o(1)$. For every vertex $v\in L(G)$, by Chernoff Bound \Cref{chernoff} with at least $1-\exp(-f(n)^2p'n_2/2)\ge 1-n^{-\omega(1)}$ probability we have $N_G(v)\ge (1-f(n))p'n_2\ge (1-o(1))d_1$. Similarly, for every vertex $v\in R(G)$ we have $N_G(v)\ge (1-o(1))d_2$. By union bound over all $v\in L(G)\cup R(G)$ and \Cref{randomsample}, with probability at least $1-o(1)$, we have $G\subseteq H$, and for all vertices $v\in L(H)$, $N_H(v)\leq N_G(v)+o(d_1)$ (Similarly for $v\in R(H)$). At the same time, $G$ has the property in the statement of \Cref{randomsample}.

Conditioned on all above, without loss of generality, we just prove the $H$ sampled in this case has the desired expansion for all size-bounded left vertex-set: Fix any $t\leq \frac{Cn_1}{d_2}$ and any set $S\subseteq L(H)$ with $|S|=t$, we know $\frac{\mathsf{UN}_H(S)}{|S|}\ge \frac{\mathsf{UN}_G(S)}{|S|}-o(d_1)\ge p'(1-p')^{t-1}n_2-\sqrt{4p'(1-p')^{t-1}n_2\log{n_1}}-o(d_1)$. It follows that
\begin{align}
&p'(1-p')^{t-1}n_2-\sqrt{4p'(1-p')^{t-1}n_2\log{n_1}}-o(d_1)\\\label{boundbiregular}
\ge&(1-o(1))d_1(1-p')^{t-1}-\sqrt{4(1-o(1))d_1(1-p')^{t-1}\log{n_1}}-o(d_1)
\end{align}
Since $\log{n_1}\leq o(d_1)$ and $d_1(1-p')^{t-1}\ge d_1\exp(-p't/(1-p'))\ge \Omega(d_1)\ge \omega(\log{n_1})$. The above can be further bounded by:
\begin{align*}
\Cref{boundbiregular}&\ge(1-o(1))d_1(1-p')^{t-1}-o(d_1)\\
&\ge(1-o(1))d_1\exp(-pt)^{1+o(1)}-o(d_1)\\
&\ge (1-o(1))d_1\exp(-pt)\exp(-C)^{o(1)}-o(d_1)\\
&\ge (1-o(1))d_1\exp(-pt)
\end{align*}
Therefore, the above is at least $(1-o(1))d_1\exp(-pt)$.
\end{proof}
Next, we generalize the above argument to $d_1=\Omega(n)$, which is crucial in the proof of \Cref{thmlosslessunique}.
\begin{theorem}\label{linearbiregular2}
For any $\delta,C>0$, there exists a constant $\epsilon>0$ such that: Let $n_1=n\leq n_2\leq \poly(n),\omega(\log{n_1})\leq d_1\leq \epsilon n_2,\omega(\log{n_2})\leq d_2\leq \epsilon n_1$, and $p=\frac{d_1}{n_2}=\frac{d_2}{n_1}$. For all sufficiently large $n$, there exists $(d_1,d_2)$-biregular graph with $n_1$ and $n_2$ left/right vertices such that:

1. For all $t\leq \frac{Cn_1}{d_2}$,  $\forall S\subseteq R(H)$, s.t. $|S|=t$, $\frac{|\mathsf{UN}_H(S)|}{|S|}\ge (1-\delta)d_2\exp(-pt)$

2. For all $t\leq \frac{Cn_2}{d_1}$,  $\forall S\subseteq L(H)$, s.t. $|S|=t$, $\frac{|\mathsf{UN}_H(S)|}{|S|}\ge (1-\delta)d_1\exp(-pt)$
\end{theorem}
\begin{proof}
Let $C_2=C_2(\delta,C),\epsilon=\epsilon(\delta,C,C_2)>0$ be some large/small enough constants to be determined later. By \Cref{embedding}, there exists a distribution $\mathsf{D}$ of $(G,H)$, whose marginal distribution on $G$ is $\mathbb{G}_{n_1,n_2,p'}$, and $H=(L\cup R,E)$ is  always some $(d_1,d_2)$ biregular bipartite graph with $|L|=n_1,|R|=n_2$. Moreover, with probability $1-o(1)$, we have $G\subseteq H$. Here $p'\ge (1-C_1\epsilon^{1/3})p$ and $C_1$ is some universal constant. By setting $\epsilon= \epsilon(\delta,C,C_2)$ small enough, we can make $p'\ge (1-\delta/(2C_2))p$.

Fix some function $\omega((\log{n}/p'n_2)^{1/2})\leq f(n)\leq o(1)$. For every vertex $v\in L(G)$. By Chernoff Bound \Cref{chernoff}, with at least $1-\exp(-f(n)^2p'n_2/2)\ge 1-n^{-\omega(1)}$ probability we have $N_G(v)\ge (1-f(n))p'n_2\ge (1-\delta/C_2)d_1$. Similarly, for every vertex $v\in R(G)$ we have $N_G(v)\ge (1-\delta/C_2)d_2$. By union bound over all $v\in L(G)\cup R(G)$ and \Cref{randomsample}, with probability at least $1-o(1)$, we have $G\subseteq H$, and for all vertices $v\in L(H)$, $N_H(v)\leq N_G(v)+\delta d_1/C_2$ (Similarly for $v\in R(H)$). At the same time, $G$ has the property in the statement of  \Cref{randomsample}.

Conditioned on all above, without loss of generality, we just prove the $H$ sampled in this case has the desired expansion for all size-bounded left vertex-set: Fix any $t\leq \frac{Cn_1}{d_2}$ and any set $S\subseteq L(H)$ with $|S|=t$, we know $\frac{\mathsf{UN}_H(S)}{|S|}\ge \frac{\mathsf{UN}_G(S)}{|S|}-\delta d_1/C_2\ge p'(1-p')^{t-1}n_2-\sqrt{4p'(1-p')^{t-1}n_2\log{n_1}}-\delta d_1/C_2$.
\begin{align}
&p'(1-p')^{t-1}n_2-\sqrt{4p'(1-p')^{t-1}n_2\log{n_1}}-\delta d_1/C_2\\\label{boundbiregular2}
\ge&(1-\delta/(2C_2))d_1(1-p')^{t-1}-\sqrt{4(1-\delta/(2C_2))d_1(1-p')^{t-1}\log{n_1}}-\delta d_1/C_2
\end{align}
Since $d_1(1-p')^{t-1}\ge d_1\exp(-p't/(1-p'))\ge \Omega(d_1)\ge \omega(\log{n_1})$, the above can be further bounded by:
\begin{align*}
\Cref{boundbiregular2}&\ge(1-o(1))(1-\delta/(2C_2))d_1(1-p')^{t-1}-\delta d_1/C_2\\
&\ge(1-\delta/C_2)d_1\exp(-p't/(1-p'))-\delta d_1/C_2\\
&\ge (1-\delta/C_2)d_1\exp(-pt)^{(1/(1-p))}-\delta d_1/C_2\\
&\ge (1-\delta/C_2)d_1\exp(-pt)\exp(-C)^{2p}-\delta d_1/C_2\\
&\ge(1-\delta/1000)d_1\exp(-pt)\exp(-C)^{2\epsilon}-\delta\exp(-C)d_1/1000   \ \text{(}C_2\ \text{large enough)}\\
&\ge(1-\delta/1000)^2d_1\exp(-pt)-\delta\exp(-C)d_1/1000\ \ \  \ \ \ \ \ \ \ \ \ \ \ \ \ \text{(}\epsilon\ \text{small enough)}\\
&\ge((1-\delta/1000)^2-\delta/1000)d_1\exp(-pt)\ \ \ \ \ \ \ \ \ \ \ \  \ \ \ \ \ \ \ \ \text{(}\exp(-pt)\ge \exp(-C)\text{)}\\
&\ge(1-\delta)d_1\exp(-pt)
\end{align*}
Therefore, the above is at least $(1-\delta)d_1\exp(-pt)$.
\end{proof}

\section*{Acknowledgements}
We thank Mahdi Cheraghchi for numerous helpful discussions including introducing \cite{lpsbook} to us, and his comments on an early version of this paper. We also thank Nikhil Shagrithaya, Yaowei Long and Thatchaphol Saranurak for insightful discussions. Besides, we thank the anonymous SODA reviewers for their detailed feedback.

The author was partially supported by the National Science Foundation under Grants No.\ CCF-2107345 and CCF-2236931.
{\small \bibliography{main}}
\appendix
\section{Subgraph Density of Biregular Graph with Bounded Spectrum}\label{spectralsec}
Our goal in this section is to prove \Cref{subgraphexpansion}. We will basically follow the proof in \cite{hmmp24}, while there are two differences: 1. Our lemma statement of \Cref{leavesbound} is different and in some sense stronger than that in the original proof. It remedies a flaw in the original paper \cite{hmmp24}. Concretely, we think the original statement of Lemma 5.7 in \cite{hmmp24} is actually not enough to derive the 4-th equation on Page21 in \cite{hmmp24}. We fix this by slightly strengthening Lemma 5.7 in \cite{hmmp24} to \Cref{leavesbound} in our paper (You can compare their statements), and complete the proof. 2. We slightly modify the bidegree requirement of \Cref{subgraphexpansion} to $2\leq c\leq d$ and $cd>6$, which is necessary for the proof of \Cref{thmtwoside}.

In this section, for any function $f\colon A\to \mathbb{R}$, we define its norm by $\Vert f\Vert_2^2=\sum_{x\in A}f(x)^2$
\begin{definition}[Biregular Tree Extension]
Fix $c,d\ge 2$ and even $\ell>1$. Given a bipartite graph $G=(L\cup R,E)$, where each left vertex has degree less $c$ and each right vertex has degree less than $d$. The $(c,d)$-biregular tree extension $T$ of $G$ is defined as follows: For each left vertex $u\in L$, attach a tree $T_u$ rooted at $u$. $T_u$ should have $\ell$ layers from counting from $0$. The $0$-th layer only contains $u$ and has $c-\dgr_G(u)$ children. Each of other vertices in odd layers  has $d$ children, and each of other vertices in even layers has $c$ children except for leaves. The similar rule also applies to each right vertex in $R$. 
\end{definition}
For any function $f\colon V(G)\to \mathbb{R}$ defined on $V(G)=L\cup R$ and its biregular tree extension $T$, given a parameter $t\in \mathbb{R}$, we can define $f_t\colon V(T)\to\mathbb{R}$ as extension of $f$ on $T$ as:
\begin{equation}
f_t(x)=f(r)t^{\mathsf{depth}(x)},\ \ \forall r\in V(G), x\in T_r
\end{equation}
Next we define Bethe-Hessian matrix of $G$, which provide a bridge between positive definiteness and subgraph denstity:
\begin{definition}
For an undirected graph $G=(V,E)$, and a parameter $t$, the Bethe-Hessian matrix is defined as $H_G(t)=(D_G-I)t^2-A_Gt+I$, where $D_G$ is the diagonal degree matrix of $G$ and $A_G$ is its adjacency matrix. 
\end{definition}
For small enough vertex set $S\subseteq V(G)$, we would like to prove positive definiteness of $H_{G[S]}(t)$. We reduce it to $H_T(t)$ by the following lemma:
\begin{lemma}\label{leafbound}
Let $G=(L\cup R,E)$ be a bipartite graph and $T$ be a $(c,d)$-biregular tree extension of it. For any $t\in\mathbb{R}$ and function $f\colon V(G)\to\mathbb{R}$, its extension $f_t\colon V(T)\to\mathbb{R}$ satisfies:
\begin{align*}
(H_T(t)\cdot f_t)(x)=
\begin{cases}
(H_G(t)\cdot f)(x)&x\in V(G)\\
0&v\notin V(G)
\end{cases}
\end{align*}
\end{lemma}
\begin{proof}
When $x\in V(G)$, we have $\dgr_G(x)$ neighbors in the original graph and $\dgr_T(x)-\dgr_G(x)$ children with value $f(x)t$, it follows that:
\begin{align*}
(H_T(t)\cdot f_t)(x)&=((\dgr_T(x)-1)t^2+1)f(x)-(\dgr_T(x)-\dgr_G(x))t^2f(x)-t(A_G\cdot f)(x)\\
&=((\dgr_G(x)-1)t^2-1)f(x)-t(A_G\cdot f)(x)\\
&=(H_G(t)\cdot f)(x)
\end{align*}
On the other hand, when $x\notin V(G)$, let $x\in T_r, d=\mathsf{depth(x)}$, we know $x$ has one parent neighbor on the tree with value $f(r)t^{d-1}$, and $\dgr_T(x)-1$ children on the tree with value $f(r)t^{d+1}$. It follows that:
\begin{equation*}
(H_T(t)\cdot f_t)(x)=((\dgr_T(x)-1)t^2+1)t^df(r)-t^df(r)-(\dgr_T(x)-1)t^{d+2}f(r)=0
\end{equation*}
\end{proof}
Next, we will define a "patched" version of degree matrix for biregular tree extension $T$ to make it "real biregular". We can observe that all vertices in a $(c,d)$-biregular tree extension $T$ have degrees $c$ or $d$ except for leaves. Therefore, we can manually "correct" these leave degrees that matches a infinite $(c,d)$-biregular tree. Fomally, $D'_T$ is modified from $D_T$ that for each left vertex $r\in L$, all leaves $x$ on tree $T_r$ have "corrected degree" $D'_T(x,x)=c$, and for each right vertex $r\in R$, all leaves $x$ on tree $T_r$ have "corrected degree" $D'_T(x,x)=d$. Note that all leaves are on even layers so this correction makes sense.

To facilitate future analysis, we need to bound $\langle f_t^{\ell},(D'_T-I)f_t^{\ell}\rangle$ compared to $\langle f_t,(D'_T-I)f_t\rangle$ in the following lemma. Here $f_t^{\ell}$ denote $f_t$ restricted to leaves of $T$, where values on leaves remain unchanged but all other vertices have values $0$.
\begin{lemma}\label{leavesbound}
Let $G=(L\cup R,E)$ be a bipartite graph where left(right) vertices have degree at most $c(d)$, and $T$ is its even depth-$\ell$ $(c,d)$-biregular tree extension. For any $f\colon V(G)\to \mathbb{R}$ and its extension $f_t\colon V(T)\to \mathbb{R}$ Let $f^{\ell}_t$ denote $f_t$ restricted to leaves of $T$. Suppose $t^2\sqrt{(c-1)(d-1)}\leq 1-\delta$ for some $\delta\in(0,1)$. Then:
\begin{equation*}
\frac{\langle f_t^{\ell},(D'_T-I)f_t^{\ell}\rangle}{\langle f_t,(D'_T-I)f_t\rangle}\leq \frac{2\delta}{e^{\delta\ell}-1}
\end{equation*}
\end{lemma}
\begin{proof}
For each $v\in V(G)$, it suffices to show the desired inequality on the single tree $T_r$. Without loss of generality, we can assume $v\in L$.

For any even $k\leq \ell$, let $f^k_t[r]$ denote the function restricted on vertices in the $k$-th layer of $T_r$. It's not hard to show that the number $n_k$ of vertices in $k$-th layer of $T_r$ can be counted by: 
\begin{equation*}
n_k=\frac{\dgr_{T_r}(r)}{(c-1)}((c-1)(d-1))^{k/2}
\end{equation*}
It follows that:
\begin{equation*}
\langle f^k_t[r],(D'_T-I)f^k_t[r]\rangle=\dgr_{T_r}(r)f(r)^2t^{2k}((c-1)(d-1))^{k/2}=\dgr_{T_r}(r)f(r)^2(1-\delta')^k
\end{equation*}
Here we define $\delta\leq \delta'=1-t^2\sqrt{(c-1)(d-1)}$. It follows that:
\begin{align*}
\frac{\langle f^{\ell}_t[r],(D'_T-I)f^{\ell}_t[r]\rangle}{\langle f_t[r],(D'_T-I)f_t[r]\rangle}=\frac{(1-\delta')^{\ell}}{\sum_{k=0}^{\ell/2}(1-\delta')^{2k}}=\frac{(1-\delta')^{\ell}(1-(1-\delta')^2)}{1-(1-\delta')^{\ell+2}}\leq \frac{2\delta'}{e^{\delta'\ell}-1}
\end{align*}
The last inequality comes from $1-x\leq e^{-x}$. By averaging argument over each $r\in V(G)$, we get:
\begin{equation*}
\frac{\langle f_t^{\ell},(D'_T-I)f_t^{\ell}\rangle}{\langle f_t,(D'_T-I)f_t\rangle}\leq \frac{2\delta'}{e^{\delta'\ell}-1}
\end{equation*}
Since the function $g(x)=\frac{2x}{e^{x\ell}-1}$ is non-increasing and $\delta\leq \delta'$, we know the above inequality still holds when we replace $\delta'$ by $\delta$.
\end{proof}
Next, we will introduce the concept of folding a tree extension into a biregular graph.
\begin{definition}
Let $G=(L\cup R,E)$ be a $(c,d)$-biregular bipartite graph, and $S\subseteq V(G)$ be a vertex set of $G$. Let $T$ be a depth-$\ell$ $(c,d)$-biregular tree extension of $G[S]$, we define a mapping $\sigma\colon V(T)\to V(G)$ iteratively as follows:

For each $v\in V(G)$ we give an order $[1,\dgr_G(v)]$ to edges incident to it. 
1. For any $v\in S$, define $\sigma(v)=v$.

2. We define $\sigma(v)$ for other vertices according to the depth order from above to bottom. For each $v\in V(T)\backslash S$, let $u$ be its parent, and $v$ is $u$'s $i$-th child. By induction we've defined $\sigma(u)=u'$, then we define $\sigma(v)=v_i$, where $(u',v_i)$ is the $i$-th \emph{untouched} edge incident to $u'$ in the order defined. Here \emph{untouched} edges mean the set of edges $(u',v')$ such that 1. Not in $E(G[S])$ if $u\in S$ 2. $v'\neq \sigma(w)$ if $u\notin S$ and $w$ is $u$'s parent. 
\end{definition}
It's intuitive that this mapping preserves the "structure" of $T$ on $G$, that no two incident edges of $T$ are mapped to the same edge in $G$. We then define a corresponding folded function with the help of $\sigma$.
\begin{definition}
Fix a mapping $\sigma\colon V(T)\to V(G)$. Given any function $f\colon V(T)\to \mathbb{R}$, and any $v\in V(G)$, we define function $f^v\colon V(T)\to \mathbb{R}$ as:
\begin{align*}
f^v(x)=
\begin{cases}
f(x)&\sigma(x)=v\\
0&\sigma(x)\neq v
\end{cases}
\end{align*}
The folded function $\widetilde{f}\colon V(G)\to\mathbb{R}$ is defined as: $\widetilde{f}(v)=\Vert f^v\Vert_2,\forall v\in V(G)$. 
\end{definition}
\begin{observation}\label{obfold}
Since the folded function preserves degrees, we know that $\langle f,D'_Tf\rangle=\langle \widetilde{f},D_G\widetilde{f}\rangle$. Moreover, $\Vert f\Vert_2=\Vert \widetilde{f}\Vert_2$
\end{observation}
\begin{lemma}\label{adjacencybound}
Use the same notations as above, and let $A_T$ and $A_G$ denote the adjacency matrices of $T$ and $G$ respectively. We have:
\begin{equation*}
\langle f,A_Tf\rangle\leq \langle \widetilde{f},A_G\widetilde{f}\rangle
\end{equation*}
\end{lemma}
\begin{proof}
It simply follows from the calculation:
\begin{align*}
\langle f,A_Tf\rangle&=\sum_{(u,v)\in E(T)}f(u)f(v)=\sum_{(u,v)\in E(G)}\sum_{\begin{subarray}\sigma(x)=u,\sigma(y)=v\\v(x,y)\in E(T)\end{subarray}}f(x)f(y)\\&\leq \sum_{(u,v)\in E(G)}\Vert f^u\Vert_2\Vert f^v\Vert_2=\langle\widetilde{f},A_G\widetilde{f}\rangle
\end{align*}
The inequality follows from Cauchy-Schwartz inequality and the property of $\sigma$ that it doesn't map incident edges in $E(T)$ to the same edge in $E(G)$.
\end{proof}
Now we are ready to show positive definiteness of $H_{G[S]}(t)$ for specific parameter ranges, which will lead to our subgraph density bound:
\begin{theorem}\label{spectralthm}
Let $\epsilon\in(0,0.1)$ and $2\leq c\leq d$ be integers where $cd>6$. Let $G=(L\cup R,E)$ be a $(c,d)$-biregular graph and $S\subseteq L\cup R$ such that $|S|\leq d^{-1/\epsilon}|L\cup R|$. Then, for any $t\ge 0$ such that:
\begin{equation}\label{conditionsize}
\frac{1}{t}\ge \sqrt{\lambda^2-(\sqrt{c-1}-\sqrt{d-1})^2}
\end{equation}
where $\lambda=\max(\lambda_2(G),\sqrt{c-1}+\sqrt{d-1})(1+5\epsilon)$

we have that $H_{G[S]}(t)$ is positive definite.
\end{theorem}
\begin{proof}
From \Cref{conditionsize}, we can derive a bound on $t^2$ which is useful for applying \Cref{leafbound}:
\begin{align}
\frac{1}{t^2}&\ge (\sqrt{c-1}+\sqrt{d-1})^2(1+2\epsilon)-(\sqrt{c-1}-\sqrt{d-1})^2\\
&\ge 4(1+\epsilon)\sqrt{(c-1)(d-1)}\\ \label{requirementleafbound}
\Rightarrow t^2&\sqrt{(c-1)(d-1)}\leq 1-\epsilon
\end{align}
In order to show positive definiteness, we would like to prove for any function $f\colon S\to \mathbb{R}$, there is $\langle f,H_{G[S]}(t)f\rangle>0$. Let $\ell=\lfloor\frac{1}{2\epsilon}\rfloor$ be an closest even integer and $T$ be the depth-$\ell$ $(c,d)$-biregular tree extension on $G[S]$, we can define extension function $f_t$, folded function $\widetilde{f}$ and other notations as previous definitions. From \Cref{leafbound}, it's easy to see that:
\begin{equation}
\langle f,H_{G[S]}(t)f\rangle=\langle f_t,H_{T}(t)f_t\rangle=\langle f_t,((D_T-I)t^2-tA_T+I)f_t\rangle
\end{equation}
\paragraph{Bound $\langle f_t,(D_T-I)t^2f_t\rangle$}: As previous defined, we can introduce $D'_T$ to correct degrees of leaves in $T$ and finally subtract the extra contribution. Therefore, it follows from \Cref{leavesbound} and \Cref{conditionsize} that:
\begin{equation}\label{firstterm}
\langle f_t,(D_T-I)t^2f_t\rangle=\langle f_t,(D'_T-I)t^2f_t\rangle-\langle f^{\ell}_t,(D'_T-I)t^2f^{\ell}_t\rangle\ge\langle f_t,(D'_T-I)t^2f_t\rangle(1-4\epsilon)
\end{equation}
The last inequality comes from $\frac{2\epsilon}{e^{\epsilon\ell}-1}\leq 4\epsilon$.
\paragraph{Bound other terms}: Given \Cref{firstterm} and combine with \Cref{obfold,adjacencybound}, we can bound the whole term by:
\begin{align}
\langle f,H_{G[S]}(t)f\rangle&\ge \langle \widetilde{f},(D_G-I)t^2\widetilde{f}\rangle(1-4\epsilon)-t\langle \widetilde{f},A_G\widetilde{f}\rangle+\Vert \widetilde{f}\Vert_2^2\\
&\ge(1-4\epsilon)\langle \widetilde{f},((D_G-I)t^2+I)\widetilde{f}\rangle-t\langle \widetilde{f},A_G\widetilde{f}\rangle
\end{align}
Let $M_G(t)=(D_G-I)t^2+I$, $\gamma_c=t^2(c-1)+1$ and $\gamma_d=t^2(d-1)+1$. Since $G$ is $(c,d)$-biregular graph, we can permuting columns and rows and write $M_G(t)$ in blocked form:
\begin{equation*}
M_G(t)=
\begin{pmatrix}
\gamma_cI&0\\
0&\gamma_dI
\end{pmatrix}
\end{equation*}
It follows that:
\begin{align*}
(1-4\epsilon)M_G(t)-tA_G&=M^{1/2}_G(t)((1-4\epsilon)I-tM^{-1/2}_G(t)A_GM^{-1/2}_G(t))M^{1/2}_G(t)\\
&M^{1/2}_G(t)((1-4\epsilon)I-\frac{t}{\sqrt{\gamma_c\gamma_d}}A_G)M^{1/2}_G(t)
\end{align*}
Then, let $g=M^{1/2}_G(t)\widetilde{f}$, it suffices to bound the following:
\begin{equation}\label{bethebound}
\langle f,H_{G[S]}(t)f\rangle\ge \langle g,((1-4\epsilon)I-\frac{t}{\sqrt{\gamma_c\gamma_d}}A_Gg\rangle=(1-4\epsilon)\Vert g\Vert_2^2-\frac{t}{\sqrt{\gamma_c\gamma_d}}\langle g,A_Gg\rangle
\end{equation}
Since $G$ is $(c,d)$-biregular, we know $A_G$ has the largest eigenvalue $\sqrt{cd}$ with multiplicity one, and the corresponding eigenvector is $\frac{1}{\sqrt{2|E(G)|}}D^{1/2}_G\mathbbm{1}$. Therefore, we can upper bound the second term by:
\begin{equation}\label{spectralbound}
\langle g,A_Gg\rangle\leq \frac{\sqrt{cd}}{2|E(G)|}\langle g,D^{1/2}_G\mathbbm{1}\rangle^2+\lambda_2\Vert g\Vert^2_2
\end{equation}
Here $\lambda_2=\max(\lambda_2(G),\sqrt{c-1}+\sqrt{d-1})$. Then, since $g$ has the same support as $\widetilde{f}$ and $\widetilde{f}$ is non-zero only on vertices in the range of $\sigma$, we can bound the first term above by bounding the size of range of $\sigma$. Since all vertices in the range of $\sigma$ must within distance $\ell$ from some vertex in $S$, we can bound it by: $|\sigma(V(G))|\leq |S|\sum_{i=0}^{\ell}d^{i}\leq |S|d^{\ell+1}\leq |E(G)|d^{-1/(3\epsilon)}$. Then, we can bound \Cref{spectralbound} by:
\begin{equation}
\frac{\sqrt{cd}}{2|E(G)|}\langle g,D^{1/2}_G\mathbbm{1}\rangle^2+\lambda_2\Vert g\Vert^2_2\leq (\frac{d^2}{2|E(G)|}|E(G)|d^{-1/(3\epsilon)}+\lambda_2)\Vert g\Vert_2^2\leq \lambda_2(1+\epsilon)\Vert g\Vert_2^2
\end{equation}
The above bound is from $|S|\leq d^{-1/\epsilon}|V(G)|,d^{-1/(4\epsilon)}\leq \epsilon$ and Cauchy-Schwartz inequality. Note that $d\ge 3$ and $\epsilon<0.01$.

Therefore, from \Cref{bethebound} and above arguments, it follows that:
\begin{equation*}
\langle f,H_{G[S]}(t)f\rangle\ge (1-4\epsilon)-\frac{t\lambda_2(1+\epsilon)}{\sqrt{\gamma_c\gamma_d}}
\end{equation*}
We would like to show that RHS of above is larger than $0$. Since $\frac{1+\epsilon}{1-4\epsilon}\leq 1+5\epsilon$, it suffices to show $t^2\lambda_2^2(1+5\epsilon)<\gamma_c\gamma_d$. It suffices to show:
\begin{equation}\label{biginequality}
\frac{1}{t^4}-\frac{1}{t^2}(\lambda^2-(c-1)-(d-1))+(c-1)(d-1)>0
\end{equation}
Since $\frac{1}{t^2}\ge \lambda^2-(c-1)-(d-1)$, it's not hard to see that LHS of \Cref{biginequality} increases as increasing of $\frac{1}{t^2}$. Therefore, in order to verify \Cref{biginequality}, it suffices to plug $\frac{1}{t^2}=\lambda^2-(c-1)-(d-1)$ into \Cref{biginequality} and verify it. Since $\frac{1}{t^4}>0$, $\frac{1}{t^2}(\lambda^2-(c-1)-(d-1))=1\leq (c-1)(d-1)$, the desired inequality \Cref{biginequality} is true, and we are done.
\end{proof}
Finally, we just need to use the reduction theorem from positive defiteness to subgraph density bound.
\begin{theorem}[\cite{hmmp24} Lemma 6.2]\label{reductionthm}
Let $G=(L\cup R,E)$ be a bipartite graph where $d_1,d_2$ denote left/right average degree of $G$ respectively. Then, for any $t\in(-1,1)\backslash\{0\}$, if $H_G(t)$ is positive definite, then $(d_1-1)(d_2-1)\leq \frac{1}{t^2}$.
\end{theorem}
Combine \Cref{reductionthm} and \Cref{spectralthm}, we prove the subgraph density bound:
\begin{corollary}[\Cref{subgraphexpansion}, Restated]
Let $2\leq c\leq d$ be integers that $cd>6$ and $0<\epsilon<0.01$. For any $(c,d)$-biregular graph $G=(L\cup R,E)$ and its vertex set $S\subseteq L\cup R$ that $|S|\leq d^{-1/\epsilon}|L\cup R|$, let $L_S=S\cap L,R_S=S\cap R, m=|E(G[S])|$ and $d_L=m/|L_S|,d_R=m/|L_R|$ denote the left/right average degree in $G[S]$, we have:
\begin{equation*}
(d_L-1)(d_R-1)\leq \lambda^2-(\sqrt{c-1}-\sqrt{d-1})^2
\end{equation*}
where $\lambda:=\max(\lambda_2(G),\sqrt{c-1}+\sqrt{d-1})(1+5\epsilon)$
\end{corollary}

\listoffixmes

\end{document}